\documentclass[a4paper]{amsart}

\usepackage[english]{babel}
\usepackage[utf8]{inputenc}
\usepackage[T1]{fontenc}
\usepackage{amsmath,amssymb,amsthm,mathtools,mathdots,nicefrac,tikz,stmaryrd,bbm,color,pgfplots}
\pgfplotsset{width=5.5cm,compat=1.9}
\usepackage{algorithm}
\usepackage[noend]{algpseudocode}
\usepackage{chngcntr}
\counterwithin{figure}{section}
\counterwithin{table}{section}

\usepackage{multirow}
\usepackage[foot]{amsaddr}

\usepackage[letterpaper, margin=1in]{geometry}

\usepackage{tikzsymbols}
\usepackage{graphicx}
\usepackage[normalem]{ulem}
\usepackage[colorinlistoftodos]{todonotes}
\usepackage[colorlinks=true, allcolors=blue]{hyperref}

\theoremstyle{definition} 
\newtheorem{theorem}{Theorem}[section]
\newtheorem{corollary}[theorem]{Corollary}
\newtheorem{lemma}[theorem]{Lemma}
\newtheorem{proposition}[theorem]{Proposition}

\newtheorem{example}[theorem]{Example}
\newtheorem{remark}[theorem]{Remark}
\newtheorem{conjecture}[theorem]{Conjecture}

\newtheorem{question}[theorem]{Question}

\newcommand{\Z}{\mathbb{Z}}

\newcommand{\one}{\mathbbm{1}}

\DeclareMathOperator{\Div}{Div}
\DeclareMathOperator{\Prin}{Prin}

\DeclareMathOperator{\outdeg}{outdeg}

\DeclareMathOperator{\val}{val}

\DeclareMathOperator{\gon}{gon}

\DeclareMathOperator{\supp}{supp}

\DeclareMathOperator{\Red}{Red}

\DeclarePairedDelimiter\abs{\lvert}{\rvert}

\DeclarePairedDelimiter\floor{\lfloor}{\rfloor}
\DeclarePairedDelimiter\norm{\lVert}{\rVert}
\DeclarePairedDelimiter\set{\{}{\}}
\DeclarePairedDelimiter\paren{(}{)}

\makeatletter
\let\oldabs\abs
\def\abs{\@ifstar{\oldabs}{\oldabs*}}
\let\oldnorm\norm
\def\norm{\@ifstar{\oldnorm}{\oldnorm*}}
\let\oldparen\paren
\def\paren{\@ifstar{\oldparen}{\oldparen*}}
\makeatother

\title{Gonality sequences of graphs}
\author{Ivan Aidun$^{1}$} 
  \address{$^{1}$UW-Madison Department of Mathematics Van Vleck Hall 480 Lincoln Drive, Madison, WI  53706}
 
 \author{Frances Dean$^{2}$}

 \address{$^{2}$Institute for Health Metrics and Evaluation, University of Washington, Seattle, WA 98121}
 
  \author{Ralph Morrison$^{3}$}
   \address{$^{3}$Department of Mathematics and Statistics, Williams College, Williamstown, MA 01267}

 \author{Teresa Yu$^{4}$}
 \address{$^{4}$Department of Mathematics, University of Michigan, Ann Arbor, MI 48109}
 
  \author{Julie Yuan$^{5}$}

 \address{$^{5}$Department of Mathematics, University of Minnesota Twin Cities, Minneapolis, MN 55455
}

\begin{document}
\maketitle

\begin{abstract}
To any graph we associate a sequence of integers called the gonality sequence of the graph, consisting of the minimum degrees of divisors of increasing rank on the graph.  This is a tropical analogue of the gonality sequence of an algebraic curve. We study gonality sequences for graphs of low genus, proving that for genus up to $5$, the gonality sequence is determined by the genus and the first gonality. We then prove that any reasonable pair of first two gonalities is achieved by some graph. We also develop a modified version of Dhar's burning algorithm more suited for studying higher gonalities.
\end{abstract}

\noindent \emph{Keywords:}  graph, gonality sequence, chip-firing, Dhar's burning algorithm

\medskip

\noindent AMS subject classifications:  14T05, 05C57, 05C85

\section{Introduction}

In \cite{bn} and \cite{bn2}, Baker and Norine introduced a theory of divisors on finite graphs, analogous to the theory of divisors on algebraic curves.  Intuitively, a divisor on a graph can be viewed as a placement of integer numbers of chips on the vertices, with two placements equivalent if they differ by certain redistributions of chips called chip-firing moves.  This theory has been extended to metric graphs \cite{gk,mz}, and has been successfully applied to prove many results on algebraic curves \cite{cdpr}.

Two important numbers associated to a divisor $D$ on a graph (or on a curve) are the \emph{degree} $\deg(D)$, which is the total number of chips; and the \emph{rank} $r(D)$.  The rank of $D$ is an integer that, in the language of chip-firing, measures the extent to which the divisor $D$ can eliminate added debt to the graph.  The \emph{divisorial gonality}, or simply the \emph{gonality}, of a graph $G$ is the minimum degree of a divisor of positive rank. For algebraic curves, gonality is defined similarly; furthermore, in this setting, an equivalent definition of gonality is the minimum degree of a non-degenerate rational map from the curve to a line.

For both graphs and curves, we can also consider higher gonalities.  For $r\geq 1$, the $r^{\text{th}}$ gonality of a graph (or curve) is the minimum degree of a rank $r$ divisor:
\[\gon_r(G)=\min\{\deg(D)\,|\,\textrm{ $D\in\textrm{Div}(G)$ and $r(D)\geq r$}\}.\]
These higher gonalities are intimately related to one of the most significant open problems for divisor theory on graphs.

\begin{conjecture}[Brill-Noether Conjecture for Graphs, \cite{baker}]\label{conjecture:bn}  Choose $g,r,d\geq 0$, and let $\rho(g,r,d)=g-(r+1)(g-d+r)$.  If $\rho(g,r,d)\geq 0$, then every graph of genus $g$ has a divisor $D$ with $r(D)=r$ and $\deg(D)\leq d$.
\end{conjecture}
In the language of higher gonalities, this conjecture can be rephrased as saying that if $G$ is a graph of genus $g$ and we have $\rho(g,r,d)\geq 0$, then $\gon_r(G)\leq d$.  This conjecture has been verified for graphs of genus at most $5$ \cite{brill}. In the special case of $r=1$, the conjecture is an upper bound on the first gonality of a graph.

\begin{conjecture}[Gonality Conjecture for Graphs, \cite{baker}]\label{conjecture:gonality} If $G$ is a graph of genus $g$, then $\gon_1(G)\leq\floor{\frac{g+3}{2}}$.
\end{conjecture}

In this paper, we study the \emph{gonality sequence} of a graph $G$, which is the sequence of $r^{\text{th}}$ gonalities of $G$ as $r$ ranges from $1$ to $\infty$:
\[\gon_1(G),\gon_2(G),\gon_3(G),\gon_4(G),\cdots\]
Recent progress has been made towards determining the gonality sequences of various families of graphs, including the complete graphs $K_n$ \cite{panizzut} and the complete bipartite graphs $K_{m,n}$ \cite{bipartite}.  Work has also been done to study higher gonalities of Erd\"{o}s-Renyi random graphs \cite{higher_gonality_random}. A very ambitious program would be to answer the following question.
\begin{question}\label{question:main}
Which integer sequences are the gonality sequence of some graph?
\end{question}
We will see in Corollary \ref{corollary:determined_genus} that the genus of a graph can be read from its gonality sequence, so an answer to Question \ref{question:main} would furnish an answer to Conjecture \ref{conjecture:bn}; because of this, we believe a complete answer would be very difficult to obtain.  In this paper, we answer the question up to genus $5$; it turns out that for such graphs, the first gonality and the genus suffice to determine the gonality sequence.

\begin{theorem}\label{theorem:gonality_sequences}
Let $G$ be a graph of genus $g$.  If $g\leq 5$, then the gonality sequence of $G$ is determined by $g$ and $\gon_1(G)$. If $g\geq 6$, then there exist graphs \(G\) and \(H\) of genus \(g\) with \(\textrm{gon}_1(G)=\textrm{gon}_1(H)\) such that \(G\) and \(H\) do not have the same gonality sequence.
\end{theorem}

All possible gonality sequences of graphs of genus at most $5$, along with an example graph for each specified gonality sequence, appear in Table \ref{table:sequences}.  That table also contains the same data for graphs of genus $6$, under the assumption that Conjecture \ref{conjecture:gonality} holds for such graphs.

We also prove that any ``reasonable'' pair of first and second gonalities are achieved by some graph.

\begin{theorem}\label{theorem:first_two}
Let $m,n\in \mathbb{Z}$ with $1\leq m< n\leq 2m$. There exists a graph $G$ such that $\gon_1(G)=m$ and $\gon_2(G)=n$.
\end{theorem}

Our paper is organized as follows.  In Section \ref{section:background}, we present general background and results for chip-firing on graphs.  In Section \ref{section:preliminary}, we present some preliminary results on gonality sequences of graphs, many arising as nice corollaries of the Riemann-Roch theorem for graphs.  In Section \ref{section:small_genus}, we prove Theorem \ref{theorem:gonality_sequences}, and also consider the possible gonality sequences of graphs of genus $6$.  In Section \ref{section:bananas}, we construct a family of graphs to prove Theorem \ref{theorem:first_two}.  We close in Section \ref{section:modified_dhars} with a modified version of Dhar's burning algorithm tailored to studying higher gonalities.

\section{Background and Notation}\label{section:background}

In this paper, all graphs will be finite, loopless, combinatorial multigraphs. The set of vertices of a graph $G$ is denoted $V(G)$, and the multiset of edges $E(G)$.  The \emph{genus} of a graph is defined as $g(G) := \abs{E(G)} - \abs{V(G)} + 1$.  If a vertex $v \in V(G)$ is an endpoint of an edge $e \in E(G)$, we use the notation $v \in e$.  We denote $E(v) \coloneqq \{ e \in E(G) \colon e \ni v \}$ and $E(u, v) \coloneqq E(u) \cap E(v)$.  Moreover, for $U, W \subset V(G)$, we denote $E(U, W) \coloneqq \displaystyle{\bigcup \set{E(u,w) \colon u \in U, w \in W}}$. The \emph{valence} of a vertex is the total number of edges incident to the vertex: $\val(v) \coloneqq \abs{E(v)}$.

\subsection{Divisor theory on graphs}

Given a graph $G$, a \emph{divisor} $D$ on $G$ is a $\Z$-linear sum of the vertices $V(G)$, and can be written as
\[
D=\sum_{v \in V(G)} D(v) \cdot v,
\]
where $D(v)\in\mathbb{Z}$ for all $v$. The set of all divisors on a graph $G$, denoted $\Div(G)$, forms an abelian group under coefficient-wise addition, namely the free abelian group generated by $V(G)$.

The \emph{degree} of a divisor $D$ is defined as the sum of its coefficients:
\[
\deg(D) \coloneqq \sum_{v \in V(G)} D(v). 
\]
The set of divisors on $G$ of a given degree $k$ is denoted $\Div^k(G)$. A divisor $D$ is said to be \emph{effective} if $D(v) \geq 0$ for all $v \in V(G)$. We use $\Div_+(G)$ to denote the set of all effective divisors on a graph $G$. The \emph{support} of an effective divisor $D \in \Div(G)$ is defined as
\[
\supp(D) \coloneqq \{ v \in V(G) \colon D(v) > 0 \}.
\]

The \emph{Laplacian matrix} of a graph $G$ is the $|V(G)|\times|V(G)|$ matrix with entries
\[
\mathcal{L}_{v,w} = 
\begin{cases}
\val(v) \quad &\text{if $v = w$} \\
-\abs{ E(v,w) } \quad &\text{if $v \neq w$}.
\end{cases}
\]
Let \(\mathcal{M}(G)=\textrm{Hom}(V(G),\mathbb{Z})\) be the abelian group consisting of all integer-valued fuctions on the vertices of \(G\); note that this may be identified with \(\textrm{Div}(G)\), and also with \(\mathbb{Z}^{|V(G)|}\).  The Laplacian operator $\Delta \colon \mathcal{M}(G) \to \Div(G)$ is then the map induced by the Laplacian matrix. We call a divisor \emph{principal} if it is in the image of $\Delta$. The set of all principal divisors on a graph $G$ is denoted $\Prin(G)$. Since the columns of $\mathcal{L}(G)$ sum to zero, $\Prin(G) \subset \Div^0(G)$. We can define an equivalence relation $\sim$ on divisors where $D \sim D'$ if and only if $D - D' \in \Prin(G)$. If this is the case, then $D$ and $D'$ are said to be \emph{linearly equivalent} divisors. It is clear that the degree of a divisor is invariant under linear equivalence. Given a divisor $D$, the \emph{linear system} associated to $D$ consists of all effective divisors linearly equivalent to $D$: 
\[
\abs{D} \coloneqq \{ D' \in \Div_+(G) \colon D' \sim D \}.
\]

We define the \emph{rank} $r(D)$ of a divisor $D$ as $r(D) = -1$ if $\abs{D} = \emptyset$ and 
\[
r(D) = \max \{ r \in \Z^+ \colon \abs{D - D'} \neq \emptyset \; \text{for all} \; D' \in \Div_+^r(G) \}
\]
otherwise. It can be easily shown that rank of a divisor is also preserved under linear equivalence.  We define the \emph{gonality} $\gon(G)$ of a graph $G$ to be
\[
\gon(G) \coloneqq \min \{ \deg(D) \colon D \in \Div_+(G), r(D) \geq 1\}.
\]

More generally, for fixed $r \in \Z_{>0}$, we define the $r^{th}$ gonality of $G$, denoted $\gon_r(G)$, as the minimum degree of a rank $r$ divisor:
\[
\gon_r(G) = \min \{ \deg(D) \colon D \in \Div(G), r(D) \geq r \}.
\]
Since the rank of a divisor can drop by at most one when a chip is removed, the minimum degree will be accomplished by a divisor of rank precisely \(r\), so this definition could equivalently have required  \( r(D) = r \).

Recall Conjecture \ref{conjecture:bn}, which says that if $G$ is a graph of genus $g$ and $\rho(g,r,d)\geq 0$, then $\gon_r(G)\leq d$. (Here $\rho(g, r, d) \coloneqq g - (r+1)(g - d + r)$.) It is proven in \cite{brill} that this conjecture holds for all graphs of genus $g \leq 5$. Hence, we have the following result, which will be useful in Section \ref{section:small_genus}.

\begin{corollary}
\label{brillnoether}\label{theorem:up_to_5}
For all graphs of genus $g \leq 5$, we have $\gon_1(G) \leq \left\lfloor \frac{g+3}{2} \right\rfloor$.
\end{corollary}




In \cite{bn}, Baker and Norine prove a Riemann-Roch theorem for combinatorial graphs, analogous to the classical Riemann-Roch theorem for algebraic curves. We restate this result here. For a graph $G$, define the \emph{canonical divisor} $K$ to be
\[
K \coloneqq \sum_{v \in V(G)} (\val(v) - 2) \cdot (v).
\]

\begin{theorem}[Riemann-Roch for graphs]\label{RR}
For a divisor $D \in \Div(G)$, 
\[
r(D) - r(K - D) = \deg(D) + 1 - g(G).
\]
\end{theorem}

As a corollary, Baker and Norine also prove a version of Clifford's theorem for graphs. We say a divisor $D$ is \emph{special} if $r(K - D) \geq 0$. 

\begin{theorem}
Let $D$ be an effective, special divisor on a graph $G$. Then
\[
r(D) \leq \frac{1}{2}\deg(D).
\]
\end{theorem}

We can similarly study a theory of divisors on metric graphs, which have lengths associated to each edge; the key difference here is that a divisor can have support at points in the interiors of edges.  Equivalence of divisors is then defined in terms of \emph{tropical rational functions}, similar to how divisor theory on algebraic curves is defined in terms of rational functions.  Once we have a notion of equivalence, we can define most terms (effective, degree, rank, gonality, the canonical divisor) in parallel, and even have a metric graph version of the Riemann-Roch theorem;  see  \cite{gk,mz} for more details.

One way to translate results between the finite and metric settings is the following: given a finite graph \(G\), associate to \(G\) a metric graph \(\Gamma\) by assigning a length of \(1\) to each edge of \(G\).  Given a divisor \(D\) on \(G\), we can consider the corresponding divisor on \(\Gamma\); it turns our that these two divisors have the same rank \cite{rank_of_divisors}.  As an example of a result that holds in both the finite and metric cases, we recall the following result from \cite{coppens} (which is also referred to as Clifford's theorem on graphs). 

\begin{theorem}
\label{coppensclif}
Let $G$ be a (finite or metric) graph of genus $g$, and assume that there exists a divisor $D$ of rank $r$ with $2\leq r\leq g-2$ such that $\deg(D)=2r$.  Then $G$ has a divisor of rank $1$ and degree $2$.
\end{theorem}

This result is proved in the metric case, although \cite[\S 1]{coppens}  presents an argument that shows it is also true for finite graphs.  Part of this is straightforward:  given a divisor of degree \(2r\) and rank \(r\) on \(G\), the corresponding divisor on the unit-length metric graph \(\Gamma\) has the same degree and rank, and so \(\Gamma\) must have a divisor of rank \(1\) and degree \(2\).  The challenge is then that arguing that \(G\) also has a divisor of rank \(1\) and degree \(2\); this is not obvious, since a divisor on \(\Gamma\) could be supported on points not corresponding to vertices of \(G\).  This portion of the argument is accomplished via a structure theorem on those metric graphs with a divisor of rank \(1\) and degree \(2\) \cite{chan}.

\subsection{Chip-Firing Games and Graph Gonality}

We now present an intuition for divisors in the language of chip configurations and chip-firing. For a divisor $D$ and a vertex $v$, we can regard $D(v)$ as representing an integer number of chips sitting on $v$, with $D(v) < 0$ indicating that $v$ is ``in debt.'' Thus, a divisor $D$ represents some configuration of chips on the graph. This will allow us to present an alternative definition of gonality in terms of the \emph{Baker-Norine chip-firing game}.

In this game, the only legal moves are \emph{chip-firing} moves. Given a vertex $v \in V(G)$, in order to chip-fire from $v$, we subtract $\val(v)$ chips from $v$ and to each $u$ adjacent to $v$ add $\abs{E(v,u)}$ chips. The Baker-Norine chip firing game is played in the following way:

\begin{enumerate}
    \item The player places $k$ chips on the vertices $V(G)$ of a graph $G$.
    \item An opponent chooses a vertex $v \in V(G)$ from which to subtract a chip.
    \item The player attempts to reach a configuration of chips where no vertex is in debt via a sequence of chip-firing moves.
\end{enumerate}

If the player can achieve a configuration of chips where every vertex is out of debt, then they win.  As discussed in \cite{bn}, chip-firing moves correspond to subtracting principal divisors, implying that the gonality of a graph is equivalent to the minimum number of chips $k$ that must initially be placed on the graph in order to guarantee a winning strategy for the first player.  Similarly, the $r^{\text{th}}$ gonality of $G$ is the minimum number of chips required to win the Baker-Norine chip-firing game if the opponent is allowed to subtract $r$ chips of their choice instead of just $1$.

Computing the gonality of a graph is difficult. For a given graph $G$, to show that $\gon(G) = k$, one must show that both of the following conditions hold:

\begin{enumerate}
	\item There exists a divisor $D$ with $\deg(D) = k$ and $r(D) \geq 1$. That is, there exists a chip configuration with $k$ chips such that no matter where the opponent subtracts a chip, the player can always reach an effective chip configuration.
    \item For \textit{all} divisors $D'$ with $\deg(D') < k$, we have $r(D') < 1$. That is, no configuration of fewer chips wins the chip-firing game.
\end{enumerate}

There are some techniques that can be used to bound the gonality of a graph. For example, one can use other invariants of the graph, such as edge-connectivity. The following lower bound on gonality is stated in \cite{liyau} and proved in \cite{gonthree}.
\begin{lemma}\label{lemma:lower_bound}
Let $\eta(G)$ denote the edge-connectivity of a graph $G$.  Then $\gon_1(G)\geq \min\{\eta(G), \abs{V(G)}\}$.
\end{lemma}





In addition, there are techniques to determine if a chip configuration is ``winning'' (that is, has positive rank). In \cite{db}, van Dobben de Bruyn presents a theory of subset-firing, along with several useful associated results. We provide a brief summary here. For a subset $A \subset V(G)$ and a vertex $v \in A$, we define the \emph{outdegree of $A$ at $v$} as
\[
\outdeg_A(v) \coloneqq \abs{E(\set{v},V(G) - A)}.
\]
The total outdegree of a subset $A$ is $\outdeg_A(A) \coloneqq \sum_{v \in A} \outdeg_A(v)$. Notice that because chip-firing moves are commutative in the Baker-Norine chip-firing game, one can fire \textit{subsets} of vertices, rather than just single vertices. In order to fire a subset $A \subset V(G)$, we send one chip along each edge $e \in E(A, V(G)-A)$. Hence, each vertex \(v\) in $A$ loses $\outdeg_A(v)$ chips. The following result is proven in \cite{db}.

\begin{lemma}
Given two effective divisors $D$ and $D'$ such that $D \sim D'$, there exists a finite sequence of subset-firing moves which transforms $D$ into $D'$ without introducing debt in any vertex of the graph.
\end{lemma}

\subsection{Dhar's Burning Algorithm}
\label{subsection:usual_dhars}

Given a vertex $v$, we say a divisor $D$ is \emph{$v$-reduced} if it satisfies two properties:
\begin{enumerate}
\item for each $v' \in V(G) - \{v\}$, $D(v') \geq 0$; and
\item for any nonempty $A \subset V(G) - \{v\}$, there exists some $v' \in A$ such that $D(v') < \outdeg_{A}(v')$.
\end{enumerate}
If $D$ only satisfies the first condition, then we say that $D$ is \emph{$v$-semi-reduced}.

It is proven in \cite[Proposition 3.1]{bn}
that given a divisor $D \in \Div(G)$ and a vertex $v \in V(G)$, there exists a unique $v$-reduced divisor $D'$ such that $D' \sim D$. Let $\Red_v(D)$ denote this unique $v$-reduced divisor. It follows that for a divisor $D \in \Div(G)$, $r(D) \geq 1$ if and only if $\Red_v(D)(v) \geq 1$ for each $v \in V(G)$.

Perhaps the most famous algorithm in chip-firing games is \emph{Dhar's burning algorithm}, named after the physicist who developed it for use in the study of sandpile models \cite{dhar}.  Given a semi-reduced divisor $D$ and a vertex $q$, this algorithm checks whether or not $D$ is $q$-reduced, and if not, outputs a subset of vertices $W\subset V(G)\setminus\{q\}$ that can be fired without introducing new debt.  By firing $W$ and iterating this algorithm, one can compute the unique $q$-reduced divisor equivalent to the semi-reduced divisor $D$; we refer to this iterated algorithm as the \emph{iterated Dhar's burning algorithm}, presented here as Algorithm \ref{originaldhars}.  Note that this algorithm recursively calls itself with  $\textbf{Alg}$.  For algorithms to $q$-semi-reduce a divisor $D$ with respect to a vertex $q$, see \cite{bs, db}.

\begin{algorithm}
\caption{Iterated Dhar's Burning Algorithm}
\label{originaldhars}
\begin{algorithmic}
	\Require vertex $q \in V(G)$, and $q$-semi-reduced divisor $D \in \Div(G)$
	\Ensure $\Red_q(D)$
	\State $W \coloneqq V(G) \setminus\set{q}$
	\While{$W \neq \emptyset$} 
	    \If{$D(v) < \outdeg_W(v)$ for some $v \in V(G)$} 
	    \State $W = W\setminus \{v\}$ \Comment{$v$ burns}
	    \Else
	        \State return $\textbf{Alg}(q, D - \Delta \mathbbm{1}_{W})$
        \EndIf
	\EndWhile
	return $D$ \Comment{entire graph burned}
\end{algorithmic}
\end{algorithm}

Given a divisor $D$, we can describe Algorithm \ref{originaldhars} in terms of a fire spreading through our graph with the chips of the divisor representing firefighters, each of whom can fight off one incoming fire. The fire starts at $q$, the vertex we are reducing with respect to.  Whenever a vertex burns, the fire spreads to all edges incident with that vertex. If a vertex has at least as many chips (firefighters) as it has incident burning edges, then the vertex is protected. Otherwise, that vertex burns as well. The fire spreads until no new vertices catch fire, at which point we chip-fire all the unburned vertices and begin this process again.  If the whole graph burns, then the divisor is $q$-reduced.

We now present an argument that this algorithm terminates and is correct; we will use similar arguments for Algorithm \ref{algorithm:modified_dhars} later in this paper. Each firing set $W$ delivered by the algorithm preserves $q$-semi-reducedness, since $D(v)\geq \outdeg_W(v)$ whenever $W$ is fired, so no new debt is ever introduced.  To see that the algorithm will eventually terminate, suppose we are given the divisor $D$, effective except perhaps at $q$, and let $D_q$ denote the (unique) $q$-reduced divisor of $D$. For $u,v\in V(G)$, denote by $d(v,u)$ the length of the shortest path from $v$ to $u$.  Let $d$ be the diameter of the graph $G$, defined as the greatest length of a shortest path between two vertices of $G$: $d=\textrm{diam(G)}:= \max_{v,u \in V(G)} d(v,u)$.  For $0 \leq i \leq d$, let $S_i = \set{v \in V(G): d(v,q) = i}$.  We define the function $\beta_q\colon \Div(G) \to \Z^{d+1}$ by
\[\beta_q(D) = \paren{\sum_{v \in S_0} D(v), \sum_{v \in S_1} D(v),\dots, \sum_{v \in S_d} D(v)}.\]
That is, the $i^{th}$ component of $\beta_q(D)$ is the number of chips in the configuration $D$ that are distance $i$ from $q$.

Note that each recursive call made by Algorithm \ref{originaldhars} involves firing a (nonempty) subset $W \subset V(G)$. Use $W_0, W_1, W_2, \dots$ to denote the consecutive subsets and label the resulting divisors $D = D_0, D_1, D_2, \dots$ such that $D_{i+1} = D_i - \Delta \mathbbm{1}_{W_i}$. We claim that if $W_i$ is nonempty, then $\beta_q(D_i) < \beta_q(D_{i+1})$ with respect to the lexicographic ordering. Pick some $w \in W_i$ such that $d(q, w)$ is minimized. Then $w$ has a neighbor $u$ such that $d(q, u) < d(q, w)$ and $u \notin W_i$ (note that we might have $u = q$). But this directly corresponds to $\beta_q(D_{i+1})$ being strictly greater than $\beta_q(D_i)$.

If Algorithm \ref{originaldhars} does not terminate, this means that there exists an infinite sequence of divisors $\{D_i\}_{i \geq 0}$ such that $\beta_q(D_i) < \beta_q(D_{i+1})$.  However, note that $\beta_q(D)$ is bounded for general $D$; in particular, each component is bounded between $\min\{0,D(q)\}$ and $\max\{\deg(D)+D(q), \deg(D)\}$.  Hence, such an infinite strictly increasing sequence cannot exist, so Algorithm \ref{originaldhars} must terminate.

%

We also need to check that the output of the algorithm is correct, i.e. that the entire graph burns only if the final configuration is $q$-reduced.  Given a divisor $D$ which is $q$-semi-reduced but not $q$-reduced,  there exists a subset $A$ not containing $q$ such that $D(v) \geq \outdeg_A(v)$ for all $v \in A$.  But  such a subset will not burn, so in particular the burning algorithm will find a subset $W$ to fire.  Thus, the whole graph burns only if the final divisor is $q$-reduced.

Since the iterated burning algorithm computes reduced divisors, it provides us with a polynomial time method for checking if the rank of a divisor $D$ is at least $1$: for each $v \in V(G)$ we compute $\Red_v(D)$ with the burning algorithm, then check that $\Red_v(D)(v) \geq 1$.  This in turn provides us with a way to compute the gonality of a graph: since there are finitely many effective divisors of degree at most $\abs{V(G)}$, we can simply check them one-by-one until we find the smallest one (as measured by degree) with positive rank.  In Section \ref{section:modified_dhars}, we will discuss the shortcomings of using the burning algorithm to try to compute higher gonalities of graphs and present an alternate algorithm designed for this setting.

\section{Preliminary Results}\label{section:preliminary}

We now consider the \emph{gonality sequence} of a graph $G$, the sequence of its $r^{\text{th}}$ gonalities for all $r\in\Z_{>0}$:
\[\gon_1(G),\gon_2(G),\gon_3(G),\ldots\]
We begin by providing some basic results on higher gonalities and gonality sequences, including bounds for the $r^{\text{th}}$ gonality of a graph based on lower gonalities. Some of these results were previously known, appearing in papers such as \cite{higher_gonality_random} or as widely known corollaries of the Riemann-Roch theorem for graphs; we include them here for completeness.

Our first result shows that gonality sequences are strictly increasing.

\begin{lemma}
\label{strict}
If $n > m$, then $\gon_n(G) > \gon_m(G)$.
\end{lemma}

\begin{proof}  
We first show that, for any given divisors $D, E \in \Div_+(G)$, 
\[
r(D - E) \geq \max \{r(D) - \deg(E), -1\}.
\]
Suppose that $r(D) - \deg(E) > -1$. (Otherwise, this inequality is clearly true.) Let $r(D) = k$. By the definition of rank, we know that $\abs{D - D'} \neq \emptyset$ for all $D' \in \Div_+^{k}(G)$. Hence, $\abs{(D - E) - E'} \neq \emptyset$ for all $E' \in \Div_+^{k - \deg(E)}(G)$ as $\abs{(D - E) - E'} = \abs{D - (E + E')}$ where $E + E' \in \Div_+^k(G)$. It follows that $r(D - E) \geq k - \deg(E) = r(D) - \deg(E)$.

Now let $\gon_n(G) = \gamma$. This means that there exists some divisor $D \in \Div_+^\gamma(G)$ such that $r(D) = n$. For any $E \in \Div_+^{n-m}(G)$, we have $r(D - E) \geq r(D) - \deg(E) = r(D) - (n-m) = m$. Since $\deg(D - E) = \gamma - (n-m) < \gamma$, we have
\[
\gon_m(G) \leq \deg(D - E) < \gon_n(G).
\]
\end{proof}

Our next result provides an upper bound on a higher gonality in terms of lower gonalities.

\begin{lemma}
\label{sum}
If $a = a_1 + \cdots + a_n$ with $a_i \in \Z_{>0}$, then $\gon_a(G) \leq \gon_{a_1}(G) + \cdots + \gon_{a_n}(G)$.
\end{lemma}

\begin{proof}
Let $\gamma_{a_i} = \gon_{a_i}(G)$ for each $i$. There exists a divisor $D_{a_i} \in \Div_+^{\gamma_{a_i}}(G)$ such that $r\left(D_{a_i}\right) = a_i$ for each $i$. This implies that $\abs{D_{a_i} - E} \neq \emptyset$ for all $E \in \Div_+^{a_i}(G)$. Let $D = D_{a_1} + \cdots + D_{a_n}$. Notice that for each $E \in \Div_+^a(G)$, there exists some decomposition $E = E_{a_1} + \cdots + E_{a_n}$ where $E_{a_i} \in \Div_+^{a_i}(G)$. 

Pick $F_i \in \Div_+(G)$ for $i \in \{1,\dots,n\}$.  Since \(F_i>0\), we have \(\sum_{i=1}^n F_i>0\), meaning $\abs{ \sum_{i=1}^n F_i } \neq \emptyset$.  
As a consequence, we see that
\[
\abs{D - E} = \abs{ \sum_{i=1}^n D_{a_i} - \sum_{i=1}^n E_{a_i} } = \abs{\sum_{i=1}^n \paren{D_{a_i} - E_{a_i} } } \neq \emptyset,
\]
because for each $i$, we have $\abs{D_{a_i} - E_{a_i}} \neq \emptyset$. We conclude that $r(D) \geq a$ which implies that $\gon_a(G) \leq \gon_{a_1}(G) + \cdots + \gon_{a_n}(G)$. 
\end{proof}

We now prove a straightforward lower bound on the $k^{\text{th}}$ gonality of a graph.

\begin{proposition}
\label{tree}
 For a graph $G$ and $k \in \Z_{>0}$, we have $\gon_k(G) \geq k$, with equality if and only if $G$ is a tree.
\end{proposition}

\begin{proof} 
Suppose we have a divisor $D \in \Div_+(G)$ with $\deg(D) < k$. Then, for all $D' \in \Div_+^k(G)$, we have $|D - D'| = \emptyset$ because $\deg(D - D') < 0$ and the degree of a divisor is preserved under linear equivalence. Hence, $r(D) < k$.  This shows that $\gon_k(G) \geq k$.

Now we show $\gon_k(G)=k$ if and only if $G$ is a tree. First we remark that $\gon_1(G)=1$ if and only if $G$ is a tree; this follows immediately from \cite[Lemma 1.1]{bn2}.  Assume that $G$ is a tree. Then $\gon_k(G)\leq k\gon_1(G)=k$ by Lemma \ref{sum}.  Combined with our lower bound of $k$, we have that $\gon_k(G)=k$.

Now assume $\gon_k(G)=k$.  Then the first $k$ terms of the gonality sequence of $G$ are
\[\gon_1(G),\gon_2(G),\cdots,\gon_{k-1}(G),k.\]
By Lemma \ref{strict}, this is a strictly increasing sequence of $k$ positive integers ending in $k$.  The only such sequence is
\[1,2,\cdots,k-1,k,\]
meaning that $\gon_1(G)=1$.  We conclude that $G$ is a tree.
\end{proof}

We now present a few well-known corollaries of Theorem \ref{RR}, the Riemann-Roch theorem for graphs.

\begin{lemma}\label{genus}  Let $G$ be a graph of genus $g$.
\begin{itemize}
    \item[(a)] We have $\gon_1(G)\leq g+1$.
    \item[(b)] 
    If $g\geq 2$, then $\gon_{g-1}(G)=2g-2$.
    \item[(c)] If $k \geq g$, then $\gon_k(G) = g + k$.
\end{itemize}
\end{lemma}

\begin{proof}

To prove part (a), let $D$ be any divisor of degree $g+1$ on $G$.  By the Riemann-Roch theorem, we have
\[r(D)=\deg(D)+1-g+r(K-D)= 2+r(K-D)\geq 1\]
since $r(K-D)\geq -1$.  As $D$ has positive rank and degree $g+1$, we have $\gon_1(G)\leq g+1$.  (We remark that this is essentially the proof of part 1 of \cite[Theorem 1.9]{bn}.)


For part (b), note that $\deg(K)=2g-2$, and $r(K)=\deg(K)-g+1+r(K-K)=2g-2-g+1+0=g-1$.  Thus $\gon_{g-1}(G)\leq 2g-2$.  Suppose for the sake of contradiction that there exists a divisor $D$ of degree $2g-3$ and rank $g-1$ or more.  Then $r(K-D)=r(D)+g-\deg(D)-1\geq g-1+g-(2g-3)-1=1$, and $\deg(K-D)=1$.  The only way $G$ can have a positive rank divisor of degree $1$ is for $G$ to be a tree, but $g\geq 1$, a contradiction.  Thus $\gon_{g-1}(G)= 2g-2$.

We now prove part (c).  Let $k \geq g$. Choose a divisor $D \in \Div_+^{g+k}(G)$. By the Riemann-Roch theorem for graphs, we know that
\begin{align*}
r(D) &= \deg(D) + 1 - g + r(K - D) \\
&= k + 1 + r(K - D) \geq k. 
\end{align*}
Hence, $\gon_k(G) \leq g + k$. Now, suppose that we have a divisor $D \in \Div_+^{g+k-\ell}(G)$ with $\ell \geq 1$ such that $r(D) \geq k$. Then, by Riemann-Roch, we see that
\begin{align*}
r(K - D) &= r(D) - (\deg(D) + 1 - g) \\
&\geq k - (g +k-\ell + 1 - g) \\
&= \ell - 1.
\end{align*}
Since $r(K - D) \geq \ell - 1$, we know that $\deg(K - D) \geq \ell - 1$. Hence,
\begin{align*}
\deg(K - D) = 2g - 2 - (g+k-\ell) &\geq \ell - 1,
\end{align*}
which implies $g - 1 \geq k$, contradicting our original assumption that $k \geq g$. We conclude that $\gon_k(G) = g + k$.
\end{proof}

This lemma helps us prove the following corollary.

\begin{corollary}\label{corollary:determined_genus}
The genus of a graph is determined by its gonality sequence.
\end{corollary}

\begin{proof}
A graph has genus zero if and only if its gonality sequence is $1,2,3,4,\ldots$ by Proposition \ref{tree}.  Moreover, we claim that a graph has genus $1$ if and only if its gonality sequence is $2,3,4,5,\ldots$. The forward direction  follows from Lemma \ref{genus}(c) with \(g=1\).  For the backward direction, any graph with genus $g\geq 2$ must have $\gon_g(G) = \gon_{g-1}(G) + 2$ by parts (b) and (c) of Lemma 3.4.

Now suppose $g\geq 2$.  By Lemma \ref{genus}, we know that $\gon_g(G) = \gon_{g-1}(G) + 2$ and $\gon_k(G) = \gon_{k-1}(G) + 1$ for all $k > g$. Thus, the genus of $G$ is the index of the last gonality that is more than one greater than the preceding gonality. We conclude that the genus of a graph is determined by its gonality sequence.
\end{proof}

Our next result says that if a graph has first gonality $2$, then its gonality sequence is determined by its genus.

\begin{proposition}\label{prop:hyperelliptic}
If $\gon_1(G) = 2$, then
\[
\gon_k(G) = \begin{cases}
2k \quad &\text{if $k < g(G)$} \\
k + g(G) \quad &\text{if $k \geq g(G)$}.
\end{cases}
\]
\end{proposition}

\begin{proof}
By part (c) of Lemma \ref{genus}, we know $\gon_k(G) = k + g(G)$ for $k \geq g(G)$. By Lemma \ref{sum}, we know that $\gon_k(G) \leq \gon_{k-1}(G) + 2$ for $2 \leq k \leq g(G)$. Since $\gon_{g(G)}(G) = 2g(G)$, this implies that $\gon_k(G) = \gon_{k-1}(G) + 2$ for $2 \leq k \leq g(G)$ so $\gon_k(G) = 2k$. 
\end{proof}

We will refer to graphs of gonality $2$ as \emph{hyperelliptic} graphs, coming from the terminology used for algebraic curves.  (Note that for algebraic curves, and in some work on divisor theory on graphs \cite{bn2}, ``hyperelliptic'' also requires \(g\geq 2\); we allow \(g=1\).) For such a graph of genus $g$, we will call its gonality sequence the hyperelliptic gonality sequence of genus $g$. We close this section with a corollary of Theorem \ref{coppensclif}.

\begin{corollary}\label{corollary:g-2}
Let $G$ be a graph of genus $g\geq 4$ with $\gon_1(G)\geq 3$.  We have $\gon_{g-2}(G)=2g-3$.
\end{corollary}

\begin{proof}
By part (b) of Lemma \ref{genus}, we know $\gon_{g-1}(G)=2g-2$, which is strictly larger than $\gon_{g-2}(G)$, so $\gon_{g-2}(G)\leq 2g-3$.  Suppose $\gon_{g-2}(G)\leq 2g-4$. Then there exists a divisor of degree $2g-4$ and rank $g-2$.  Since $g\geq 4$, we have $2\leq g-2$, so by Theorem \ref{coppensclif} it follows that $G$ has a divisor of rank $1$ and degree $2$, a contradiction to $\gon_1(G)\geq 3$.  We conclude that $\gon_{g-2}(G)=2g-3$.
\end{proof}

\section{Gonality Sequences for Graphs of Small Genus}\label{section:small_genus}

In this section, we will determine all possible gonality sequences for graphs of low genus. Moreover, we will show that for low genus, the genus of the graph combined with first gonality is enough information to determine the entire gonality sequence.

We first present the following lemma, which handles the $g\geq 6$ part of Theorem \ref{theorem:gonality_sequences}.

\begin{lemma}\label{lemma:G_and_H}
For $g\geq 6$, there exist graphs $G$ and $H$ of genus $g$ with $\gon_1(G)=\gon_1(H)=3$, $\gon_2(G)=5$, and $\gon_2(H)=6$.
\end{lemma}

\begin{proof}
We explicitly construct our graphs as follows.  Let $G$ have $g$ vertices $v_1,\cdots,v_g$, with three edges connecting $v_1$ and $v_2$ and two edges connecting $v_i$ and $v_{i+1}$ for $2\leq i\leq g-1$.  Let $H$ have $g-1$ vertices $w_1,\cdots,w_{g-1}$, with three edges connecting $w_1$ and $w_2$, three edges connecting $w_2$ and $w_3$, and two edges connecting $w_i$ and $w_{i+1}$ for  $3\leq i\leq g-2$.  These graphs are illustrated in Figure \ref{figure:G_and_H} in the case of $g=6$, along with certain divisors.

\begin{figure}[hbt]
    \centering
    \includegraphics[scale=1]{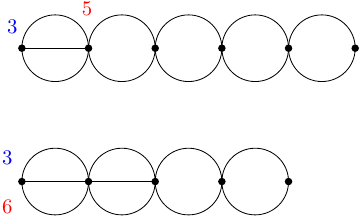}
    \caption{The graphs $G$ and $H$ for $g=6$, along with divisors of ranks $1$ and $2$}
    \label{figure:G_and_H}
\end{figure}

We first remark that neither graph has first gonality $2$: if a rank $1$ divisor has only $2$ chips, then it must place $1$ chip on the leftmost divisor since no chips can be moved there without introducing debt. But this only leaves one chip for the rest of the graph, which is not sufficient. The illustrated divisors of degree $3$, namely $3\cdot (v_1)$ on $G$ and $3\cdot (w_1)$ on $H$, do each have rank $1$, so $\gon_1(G)=\gon_2(H)=3$. Moreover, we claim that the divisor $5\cdot (v_2)$ of degree $5$ on $G$ has rank $2$. To see this, we note that for any $i,j\geq 2$ the divisor $5\cdot(v_2)$ is equivalent to the divisors $3\cdot(v_1)+2\cdot(v_i)$ and $(v_2)+2\cdot(v_i)+2\cdot(v_j)$.  This means that for any effective divisor of degree $2$, $5(v_2)$ is equivalent to some divisor greater than it, meaning that $r(5\cdot(v_2))\geq 2$.
Since $G$ is not hyperelliptic, we have $\gon_2(G)\neq 4$ by Theorem \ref{coppensclif}, so $\gon_2(G)=5$.  Although we have assumed that $g\geq 6$, note that this same argument works to argue that $\gon_2(G') = 5$ for any graph $G'$ constructed in the same manner as $G$ with $g(G') \geq 3$. 

It remains to show that $\gon_2(H)=6$.  Certainly $\gon_2(H)\leq 2\gon_1(H)=6$; indeed, we obtain a rank $2$ divisor $6\cdot(w_1)$ of degree $6$ by doubling our rank $1$ divisor $3\cdot(w_1)$.  Suppose for the sake of contradiction that there exists a rank $2$ divisor $D$ of degree $5$ on $H$.  Without loss of generality, since $D$ has rank $2$, we may assume that $D(w_1)\geq 2$.  We will now deal with several cases, where all chip-firing moves we consider are subset-firing moves that do not introduce debt.

\begin{itemize}
    \item Suppose $D(w_1)=2$.  Since $\val(w_1) > 2$, these two chips cannot be moved from $w_1$. Hence, we are left with three chips for the rest of the graph, which is equivalent to a copy of $G$ with genus $g-2\geq 4$. But as we have already shown, such a graph needs at least $5$ chips to win the second gonality game.

    \item Suppose $D(w_1)=3$, so that it has two chips elsewhere on the graph. If we remove one chip from $w_1$, the remaining two chips on $w_1$ are now isolated. We are now in the situation of playing the first gonality game on a copy of $G$ with genus $g-2\geq 4$ and only two chips.  However, we need at least three chips to win the first gonality game on $G$.
    
    \item Suppose $D(w_1)\in\{4,5\}$. Chip-fire $w_1$ to produce a linearly equivalent divisor where any chips remaining on $w_1$ are isolated and the rest of the graph has at most four chips. We are now back in the same situation as in the first case.
\end{itemize}
In all cases we have reached a contradiction, so we must have $\gon_2(H)=6$, as desired.
\end{proof}

We are now ready to prove Theorem \ref{theorem:gonality_sequences}.

\begin{proof}[Proof of Theorem \ref{theorem:gonality_sequences}]  Let $G$ be a graph with $g(G)\leq 5$.  We will prove genus by genus that $\gon_1(G)$ determines the gonality sequence of $G$.

For $g\leq 2$, we will see that the gonality sequence is determined by $g$ alone without needing to know $\gon_1(G)$.  If $g(G)=0$, then $G$ is a tree, and so has gonality sequence given by $\gon_k(G)=k$ by Proposition \ref{tree}.  If $g=1$, then $\gon_k(G) = k+1$ by the proof of Corollary \ref{corollary:determined_genus}. If $g = 2$, then by part (b) of Lemma \ref{genus}, $\gon_1(G) = 2$ so the gonality sequence of $G$ is the hyperelliptic sequence of genus $2$ given by Proposition \ref{prop:hyperelliptic}.

Now suppose $g=3$.  Since $G$ is not a tree, we know $\gon_1(G)>1$, and by Theorem \ref{theorem:up_to_5}, we have $\gon_1(G)\leq\floor{\frac{g+3}{2}}=3$.  If $\gon_1(G)=2$, then $G$ has the hyperelliptic gonality sequence of genus $3$ by Proposition \ref{prop:hyperelliptic}.  If $\gon_1(G)=3$, then for $k\geq 3$ we have $\gon_k(G)=k+3$ by part (c) of Lemma \ref{genus}.  By part (b) of Lemma \ref{genus}, $\gon_2(G) = 4$. We conclude that the gonality sequence for a graph is genus $3$ is either
\[2,4,6,7,8,9,\ldots\qquad\text{or}\qquad3,4,6,7,8,9,\ldots\]

Now suppose $g=4$.  Again we have $1 < \gon_1(G) \leq 3$. If $\gon_1(G)=2$, then $G$ has the hyperelliptic gonality sequence of genus $4$ by Proposition \ref{prop:hyperelliptic}.  If $\gon_1(G)=3$, then by Lemma \ref{genus} we have $\gon_k(G)=k+4$ for $k\geq 4$, and $\gon_3(G)=2g-2=6$.  Moreover, we can apply Corollary \ref{corollary:g-2} to deduce that $\gon_2(G)=2g-3=5$.  We conclude that the gonality sequence of a graph of genus $4$ is either
\[2,4,6,8,9,10,\ldots\qquad\text{or}\qquad3,5,6,8,9,10,\ldots\]

Finally, suppose $g=5$.  We have $\gon_1(G)>1$ and $\gon_1(G)\leq\floor{\frac{g+3}{2}}=4$.   If $\gon_1(G)=2$, then $G$ has the hyperelliptic gonality sequence of genus $4$ by Proposition \ref{prop:hyperelliptic}. Now assume $\gon_1(G)\geq 3$. We know $\gon_4(G)=8$ by part (b) of Lemma \ref{genus}, and since $G$ is not hyperelliptic, we have $\gon_3(G)=7$ by Theorem \ref{corollary:g-2}, so only the first two gonalities are yet to be determined. Since $G$ is not hyperelliptic, we know that $\gon_2(G) > 4$. By applying the Riemann-Roch theorem, we find that $G$ has a divisor $D$ of degree $3$ and rank $1$ if and only if it also has a divisor $K-D$ of degree $5$ and rank $2$. Combined with the fact that $5\leq \gon_2(G)<\gon_3(G)=7$, we have that if $\gon_1(G)=3$, then $\gon_2(G)=5$; and if $\gon_1(G)=4$, then $\gon_2(G)=6$.  We conclude that the only possible gonality sequences for a graph of genus $5$ are
\[2,4,6,8,10,11,\ldots,\qquad3,5,7,8,10,11,\ldots,\qquad\text{and}\qquad4,6,7,8,10,11,\ldots\]
Hence for all $g\leq 5$, the genus and the first gonality of a graph determine the gonality sequence. As shown in Lemma \ref{lemma:G_and_H}, this does not hold for $g\geq 6$.  This completes the proof.
\end{proof}

\begin{remark}
We remark that each possible gonality sequence discussed in the proof of Theorem \ref{theorem:gonality_sequences} is indeed the gonality sequence of some graph; see Table \ref{table:sequences} for examples corresponding to each sequence. For $g \leq  2$, any graph of genus $g$ will suffice since there is a unique possible gonality sequence for that genus.  For $g \geq 3$, the hyperelliptic gonality sequence of genus $g$ is achieved by the \emph{banana graph}, comprised of $2$ vertices connected by $g+1$ edges. The graphs $K_4$ and $K_{3,3}$ both have first gonality $3$ and thus achieve the remaining possible gonality sequences for graphs of genus $3$ and $4$, respectively. Finally, for genus $5$, the gonality sequence beginning with first gonality $3$ is achieved by the genus $5$ version of the graph $G$ used in Lemma \ref{lemma:G_and_H}. The final gonality sequence beginning with first gonality $4$ is achieved by the graph illustrated in Table \ref{table:sequences} (this graph has first gonality $4$ by \cite[Proposition 4.5]{gonthree}). This table also includes all known gonality sequences for graphs of genus $6$, as computed in the following example. Assuming the gonality conjecture holds for graphs of genus $6$, this is a complete enumeration of the gonality sequences of graphs of genus $6$.
\end{remark}

\begin{example}\label{example:g6}
In this example, we determine all possible gonality sequences of graphs of genus $6$, under the assumption that any such graph has gonality at most $\floor{\frac{g+3}{2}}=4$.  First if $\gon_1(G)=2$, then $G$ has the hyperelliptic gonality sequence of genus $6$ from Proposition \ref{prop:hyperelliptic}.  If $\gon_1(G)\geq 3$, then by Lemma \ref{genus} and Corollary \ref{corollary:g-2}, the gonality sequence of the graph is of the form
\[\gon_1(G),\gon_2(G),\gon_3(G),9,10,12,13,14,\ldots\]
Since $\gon_1(G) \geq 3$, we have $5 \leq \gon_2(G) < \gon_3(G)$ and $7 \leq \gon_3(G) < 9$. By the Riemann-Roch theorem for graphs, a graph of genus $6$ has a divisor of degree $3$ and rank $1$ if and only if it has a divisor of degree $7$ and rank $3$. Thus, if $\gon_1(G) = 3$, we have $\gon_3(G) = 7$ and so our gonality sequence has the form
 \[3,\gon_2(G),7,9,10,12,13,14,\ldots\]
 where $\gon_2(G)\in\{5,6\}$.  The example graphs $G$ and $H$ of genus $6$ from Lemma \ref{lemma:G_and_H} have first gonality $3$ and second gonalities $5$ and $6$, respectively, so both sequences are achieved by some graph.
 
 On the other hand, if $\gon_1(G)=4$, we have $\gon_3(G)=8$, so the gonality sequence is of the form
  \[4,\gon_2(G),8,9,10,12,13,14,\ldots\]
By the Riemann-Roch theorem for graphs, a graph of genus $6$ has a divisor of degree $4$ and rank $1$ if and only if it has a divisor of degree $6$ and rank $2$, so $5\leq \gon_2(G)\leq 6$.  By \cite[Theorem 1]{panizzut}, the complete graph $K_5$ has genus $6$, first gonality $4$, and second gonality $5$. We can also construct a genus $6$ graph $G$ with first gonality $4$ and second gonality $6$ as follows:  start with four vertices $v_1,v_2,v_3,v_4$, and connect $v_1$ and $v_2$ by two edges, $v_2$ and $v_3$ by five edges, and $v_3$ and $v_4$ by two edges.  A similar argument to that of Lemma \ref{lemma:G_and_H} shows that $\gon_1(G)=4$ and $\gon_2(G)=6$.  Thus both sequences are achieved by some graph.

These sequences and example graphs achieving them appear in Table \ref{table:sequences}.  If we do not assume the gonality conjecture holds, then conceivably
\[5,7,8,9,10,12,13,14,\ldots\qquad\text{and}\qquad6,7,8,9,10,12,13,14,\ldots\]
could be gonality sequences for graphs of genus $6$.  No other sequence would be possible, since $\gon_1(G)>4$ implies $\gon_2(G)>6$, and since the sequences must be increasing and stabilize at $9,10,12,13,14,\ldots$.  Thus, to prove that the gonality conjecture holds for graphs of genus $6$, it would suffice to show that neither of the above sequences is realized as the gonality sequence of a graph of genus $6$.

We also remark that the Brill-Noether conjecture in the case of $g=6$, $r=2$, and $d=6$ says that any graph of genus $g = 6$ should have $\gon_2(G)\leq 6$.  Thus the gonality conjecture for genus $6$ (i.e. the Brill-Noether conjecture for $g=6$, $r=1$, and $d=4$) is equivalent to the Brill-Noether conjecture for $g=6$, $r=2$, and $d=6$.
\end{example}

\begin{table}
\begin{center}
\begin{tabular}{|c|c|c|}
    \hline
    Genus & Gonality sequence & Example graph\\
    \hline
    $0$ & $1,2,3,4,5,6,\ldots$ & \includegraphics{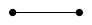} \\ \hline
    $1$ & $2,3,4,5,6,7,\ldots$ & \includegraphics{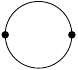} \\ \hline
    $2$ & $2,4,5,6,7,8,\ldots$ & \includegraphics{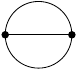} \\
    \hline
    \multirow{2}{*}{$3$}&$2,4,6,7,8,9,\ldots$&\includegraphics{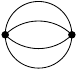}\\
    \cline{2-3}
    &$3,4,6,7,8,9,\ldots$&\includegraphics{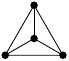}\\
    \hline
    \multirow{2}{*}{$4$}&$2,4,6,8,9,10,\ldots$&\includegraphics{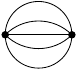}\\
    \cline{2-3}
    &$3,5,6,8,9,10,\ldots$&\includegraphics{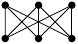}\\
    \hline
\end{tabular}\quad \begin{tabular}{|c|c|c|}
    \hline
    Genus & Gonality sequence & Example graph\\
    \hline
    \multirow{3}{*}{$5$}&$2,4,6,8,10,11,\ldots$&\includegraphics{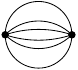}\\
    \cline{2-3}
    &$3,5,7,8,10,11,\ldots$&\includegraphics{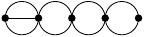}\\
    \cline{2-3}
    &$4,6,7,8,10,11,\ldots$&\includegraphics{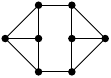}\\
    \hline
    \multirow{5}{*}{6}&$2,4,6,8,10,12,\ldots$&\includegraphics{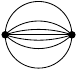}\\
    \cline{2-3}
    &$3,5,7,9,10,12,\ldots$&\includegraphics{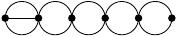}\\
    \cline{2-3}
    &$3,6,7,9,10,12,\ldots$&\includegraphics{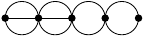}\\
    \cline{2-3}
    &$4,5,8,9,10,12,\ldots$&\includegraphics{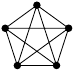}\\
    \cline{2-3}
    &$4,6,8,9,10,12,\ldots$&\includegraphics{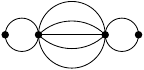}\\
    \hline
\end{tabular}

\caption{All possible gonality sequences for graphs of genus at most $5$, with example graphs having those gonality sequences; and the same for genus $6$, assuming the gonality conjecture holds for graphs of genus $6$.}
\label{table:sequences}
\end{center}
\end{table}

We now compare and contrast the properties of our gonality sequences of graphs with those of algebraic curves. The analogue of Proposition \ref{prop:hyperelliptic} holds for algebraic curves of gonality $2$, also known as hyperelliptic curves: the gonality sequence of the curve is determined by the genus of the curve and is given by the same formula. However, the behavior of graphs and curves of gonality $3$ do not so closely  mirror one another. We recall the following formula from algebraic geometry \cite[Remark 4.5]{trigonal}: if $C$ is a smooth projective curve of gonality $3$ and genus $g\geq 4$, then its $k^{\text{th}}$ gonality is
\[
\gon_k(C) = 
\begin{cases}
3k \quad &\text{if $1 \leq k \leq \left\lfloor \frac{g-1}{3} \right\rfloor$} \\
g + k - 1 - \left\lfloor \frac{g-k-1}{2} \right\rfloor \quad &\text{if $\left\lfloor \frac{g-1}{3} \right\rfloor < k \leq g - 1$} \\
g + k \quad &\text{if $k \geq g$}.
\end{cases}
\]
In particular, if a curve $C$ has gonality $3$, then its gonality sequence is determined by its genus.  The same does not hold for graphs of gonality $3$, as illustrated by Lemma \ref{lemma:G_and_H}. Thus, certain sequences which are the gonality sequence of a graph are \emph{not} the gonality sequence of any smooth projective curve. For instance, by Example \ref{example:g6}, there exists a graph $G$ of genus $6$ with gonality sequence
\[3,5,7,9,10,12,13,14,\ldots\]
 If a smooth projective algebraic curve were to have this gonality sequence, it must also have genus $6$, by the analogue of Corollary \ref{corollary:determined_genus} for curves. But any such curve with first gonality $3$ must have gonality sequence \[3,6,7,9,10,12,13,14,\ldots\] by the above formula.  Thus the gonality sequence of $G$ is not the gonality sequence of any smooth projective curve.  A similar argument applies for the genus $g$ version of the graph $G$ from Example \ref{lemma:G_and_H} for any $g\geq 7$:  any trigonal algebraic curve of genus $g$ must have second gonality $6$, but $G$ has second gonality $5$.  Thus for every genus $g\geq 6$, there exists a gonality sequence of a graph that is not the gonality sequence of a smooth algebraic curve.  However, all graph gonality sequences up to genus $5$ are also gonality sequences of curves:  there certainly exist curves of all possible genus and first gonality pairs, and the remainder of each sequence is determined by similar Riemann-Roch and Clifford's theorem arguments for algebraic curves.

We close this section with a remark on a known result for higher gonalities of algebraic curves which remains open for graphs.

\begin{lemma}[Lemma 4.3 in \cite{trigonal}]\label{lemma:adding_algebraic}  Let $C$ be a smooth projective curve of genus $g\geq 4$, such that $\textrm{gon}_{r+s}(C)=\textrm{gon}_{r}(C)+\textrm{gon}_{s}(C)$ for some $r,s\in\Z$.  Then $\textrm{gon}_k(C)=k\gon_1(C)$ for $1\leq k\leq r+s$.
\end{lemma}

We pose the analogous question for graphs.

\begin{question}\label{question:adding_graphs}
Let $G$ be a graph, and suppose $\gon_{r+s}(G)=\gon_r(G)+\gon_s(G)$.  Is it true that $ \gon_k(G)=k\gon_1(G)$ for $1\leq k\leq r+s$?
\end{question}

Note that the answer to this question is affirmative for $g(G)\leq 5$, as can be verified by examining all gonality sequences in Table \ref{table:sequences}.  We also remark that such a criterion would be useful in determining the first three gonalities of a graph; for instance, it would rule out $(2,3,5)$ and $(3,5,8)$ as possible starts to a gonality sequence.  The first of these is impossible by Proposition \ref{prop:hyperelliptic}; the authors do not know whether the second is possible.

The proof of Lemma \ref{lemma:adding_algebraic} in \cite{trigonal} uses the key fact that given two divisors $D$ and $E$ on a curve such that $r(D+E)=r(D)+r(E)$, we have $D\sim r(D)F$ and $E\sim r(E)F$ for some divisor $F$ of rank $1$; see \cite[Lemma 1.8]{4-gonal}.  Unfortunately such an argument will not work for graphs, as shown in the following example.  Thus a new approach will be needed to answer Question \ref{question:adding_graphs}.

\begin{example}\label{example:payne}
Let $G$ be the graph pictured in Figure \ref{fig:payne}.  By the proof of Theorem 1.2 in \cite{loop_of_loops}, each divisor  $D_i=v_1+v_2+w_i$ where $1\leq i\leq 3$ has rank $1$.  (Although this result is proven for metric graphs, the same holds here in the finite case.)  Since $D_1$, $D_2$, and $D_3$ are pairwise linearly inequivalent, we may choose two of them, say $D_i$ and $D_j$ such that $D_i\not\sim K-D_j$.  Then we have
\[r(D_i+D_j)-r(K-D_i-D_j)=6+1-4=3.\]
Note that $\deg(K-D_i-D_j)=0$, and since $D_i\not\sim K-D_j$, we have $r(K-D_i-D_j)=-1$.  Thus $r(D_i+D_j)=2=r(D_i)+r(D_j)$.  However, since $D_i\not\sim D_j$, there cannot exist a divisor $F$ with $D_i\sim F$ and $D_j\sim F$.  Thus the natural graph-theoretic analogue of \cite[Lemma 1.8]{4-gonal} does not hold.

\begin{figure}[hbt]
\begin{center}
\includegraphics[scale=0.8]{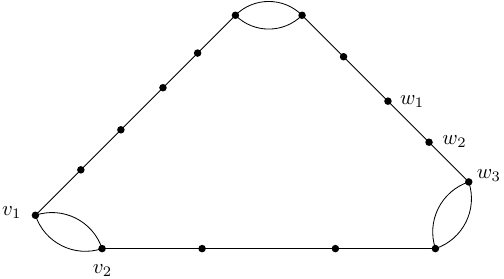}
\caption{The graph $G$ from Example \ref{example:payne}}
\label{fig:payne}
\end{center}
\end{figure}

\end{example}

\section{Banana Graphs and Second Gonalities}\label{section:bananas}

The main goal of this section is to prove that given any pair of ``reasonable'' integers $m$ and $n$, there exists a graph $G$ with first gonality $m$ and second gonality $n$.  Certainly a necessary condition is that $m<n\leq 2m$; we will show that this condition is sufficient by constructing families of graphs called the generalized banana graphs.  There are certainly other graphs that can achieve these gonalities as well; for example, complete graphs \(K_m\) achieve first gonality \(m\) and second gonality \(m+1\) \cite{panizzut}, while complete bipartite graphs of the form \(K_{m,m}\) achieve first gonality \(m\) \cite{db} and second gonality \(2m-1\) \cite{bipartite}.

As in the previous section, we define the \textit{banana graph} $B_n$ as a graph with two vertices and $n$ edges connecting them. It is clear that for $n \geq 2$, $\gon_1(B_n) = 2$, since $\gon_1(G)\leq |V(G)|$ and $\gon_1(G)>1$ for any graph that is not a tree.  It follows immediately from Proposition \ref{prop:hyperelliptic} and the fact that $g(B_n)=n-1$ that $\gon_k(B_n)=2k$ for $1\leq k \leq n-1$.


A \textit{generalized banana graph} is a graph with vertices $\{v_1,\dots,v_n\}$ such that $v_i$ is connected to $v_{i+1}$ by at least one edge, and such that there are no other edges present in the graph.  In other words, it is a collection of $n-1$ banana graphs, glued together in a line.  In this section we study two families of generalized banana graphs:  those with a constant number of edges between pairs of vertices, and those with a descending number of edges.

Let $B_{n,e}$ denote the generalized banana graph on $n$ vertices where $|E(v_i,v_{i+1})|=e$ for $1\leq i\leq n-1$. Note that $B_{n,e}$ has edge-connectivity $\eta(B_{n,e})=e$, so by Lemma \ref{lemma:lower_bound} we have 
\[\gon_1(B_{n,e}) \geq \min\{\eta(B_{n,e}),|V(B_{n,e})|\}= \min \{e, n\}.\]  
The divisors $e(v_1)$ and $(v_1)+\cdots+(v_n)$ both have positive rank, so  $\gon_1(B_{n,e}) =  \min \{e, n\}$.

\begin{example}
The generalized banana graphs $B_{4,2}$ and $B_{5,5}$ are illustrated in Figure~\ref{fig:genbananas}.
\end{example}

\begin{figure}[H]
\begin{center}
\includegraphics[scale=0.8]{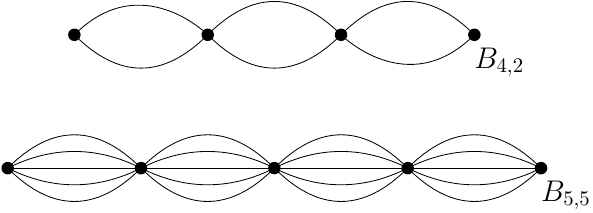}
\caption{The generalized banana graphs $B_{4,2}$ and $B_{5,5}$.}
\label{fig:genbananas}
\end{center}
\end{figure}

Over the following three lemmas, we will determine the second gonality of $B_{e,n}$ for any values of $e$ and $n$.  We will see that these graphs will not give us all the desired pairs of first and second gonalities, meaning we will need to consider other generalized banana graphs.

\begin{lemma}
If $n<e$, then $\gon_2(B_{e,n})=2n$.
\end{lemma}

\begin{proof}
By Lemma \ref{sum}, $\gon_2(B_{n,e}) \leq 2\min\{n, e\} = 2n$. Now, suppose that we have a divisor $D \in \Div_+(B_{n,e})$ with $\deg(D) = 2n-1$. We will show that $r(D)<2$. We split into the following cases.
\begin{itemize}
    \item[(1)] If $\supp(D) = V(B_{n,e})$, then there exists at least one vertex $v_i$ with $D(v_i) = 1$. Furthermore, for all other vertices $v_j$, $D(v_j) \leq n < e$. Hence, the divisor $D - 2 (v_i)$ is $v_i$-reduced, implying that $r(D) < 2$.
    \item[(2)] If exactly one vertex $v_i$ has zero chips and all other vertices have at least one chip, then there are two cases.
    \begin{itemize}
        \item[(i)] There exists exactly one vertex $v_j$ with $e$ chips. Note that since $2n - 1 - e \leq n - 2$, $v_j$ clearly cannot have more than $e$ chips and since $2n - 1 - 2e < 0$, we cannot have more than one vertex with $e$ chips. If $v_j$ has $e$ chips, then every other vertex (except for $v_i$ and $v_j)$ must have exactly one chip. Hence, the divisor $D - (v_i) - (v_j)$ is $v_i$-reduced.
        \item[(ii)] All vertices have fewer than $e$ chips. Notice that in this case, $D - 2(v_i)$ is $v_i$-reduced.
    \end{itemize}
    Both outcomes imply that $r(D)<2$.
    \item[(3)] If there are at least two distinct vertices $v_i$ and $v_j$ with $D(v_i)=0=D(v_j)$, then there is either exactly one other vertex $v_k$ such that $D(v_k)\geq e$, or there is no such vertex. Suppose such a $v_k$ exists, and consider the divisor $D'=D-(v_i)-(v_j)$. Then we fire subsets of vertices in order to move $e$ chips from $v_k$ to either $v_i$ or $v_j$, whichever is the closer of the two to $v_k$; without loss of generality assume that it is \(v_i\). Then, we can run Dhar's burning algorithm on this new configuration with respect to $v_j$, and the whole graph burns. If no such $v_k$ exists, then there is already no way to fire any subset of vertices without introducing debt. It follows that $r(D)<2$. 
\end{itemize}
Thus, $\gon_2(B_{n,e})=2n$.
\end{proof}

\begin{lemma}
If $e<n$, then $\gon_2(B_{n,e})=2e$.
\end{lemma}

\begin{proof}
By Lemma \ref{sum}, $\gon_2(B_{n,e}) \leq 2 \min \{n, e \} = 2e$. Suppose that we have a divisor $D \in \Div_+(B_{n,e})$ such that $\deg(D) = 2e - 1$.
Again, we proceed by cases to show that $r(D)<2$.
\begin{itemize}

    \item[(1)] If $\supp(D) = V(B_{n,e})$, at least one vertex $v_i$ has exactly one chip. Furthermore, suppose some vertex $v_j$ has at least $e$ chips. Then, we have $2e - 1 - e = e - 1 < n-1$ chips remaining for the $n-1$ vertices, a contradiction. Thus, if we run Dhar's burning algorithm on the divisor $D - 2(v_i)$, then the entire graph burns because $D(v) < e$ for all vertices $v$. Thus $r(D)<2$.
    
    \item[(2)] If there is exactly one vertex $v_i$ satisfying $D(v_i)=0$, we have two cases.
    \begin{itemize}
        \item[(i)] If there exists a vertex $v_j$ such that $D(v_j) \geq e$, we know that $v_j$ must have exactly $e$ chips because $2e - 1 - (e + 1) = e - 2 < n - 2$, which is a contradiction to all vertices besides $v_i$ having a chip. Furthermore, all other vertices except for $v_i$ and $v_j$ must have exactly one chip. This is because we have $e - 1 \leq n-2$ chips remaining for $n-2$ vertices. Run Dhar's burning algorithm on the the divisor $D - (v_i) - (v_j)$ beginning at the vertex $v_i$. Notice that in this divisor, no vertex has more than $e-1$ chips so the whole graph burns.
        \item[(ii)] If all vertices have fewer than $e$ chips, then we can run Dhar's burning algorithm on the divisor $D - 2(v_i)$, beginning at the vertex $v_i$, which burns the whole graph.
    \end{itemize}
    In both cases we can conclude that $r(D)<2$.
    \item[(3)] If there are at least two vertices, $v_i$ and $v_j$ with zero chips, then consider the divisor $D-(v_i)-(v_j)$. There is either exactly one other vertex $v_k$ such that $D(v_k)\geq e$, or there is no such vertex. Suppose such a $v_k$ exists. Then we can fire subsets of vertices in order to move $e$ chips from $v_k$ to $v_i$, without loss of generality. Then, we can run Dhar's burning algorithm on this configuration with respect to $v_j$, and the whole graph burns. If no such $v_k$ exists, then there is no way to fire any subset of vertices without introducing debt.  This gives $r(D)<2$.
\end{itemize}
Thus, $\gon_2(B_{n,e})=2e$.
\end{proof}

\begin{lemma}
If $n=e$, then $\gon_2(B_{n,e})=2n-1.$
\end{lemma}

\begin{proof}
First note that the divisor $n(v_1)+(v_2)+\cdots +(v_n)$ wins the second gonality game. If the opponent were to take away two chips from a vertex that currently has one, then we could fire $v_1$, and then increasingly larger subsets, in order to move the $n$ chips and reach an effective divisor. Now suppose that divisor $D$ wins the second gonality game, with $\deg(D)<2n-1$. We proceed by cases.
\begin{itemize}
\item[(1)] $\supp(D)=V(B_{n,e})$, in which case, there exist at least two vertices, $v_i$ and $v_j$, such that $D(v_i)=1=D(v_j)$.
\item[(2)] There exists some vertex $v_i$ such that $D(v_i)=0$, and $V-\left\{v_i\right\}=\supp(D)$.
\item[(3)] There exist at least two vertices $v_i$ and $v_j$ such that $D(v_i)=0=D(v_j)$, with $i<j$.
\end{itemize}

If we are in case (1), then consider the divisor $D-2(v_i)$. Notice that no vertex has greater than $n-1$ chips, so if we run Dhar's burning algorithm with respect to $v_i$, then the entire graph burns. If we are in case (2), then the maximum number of chips a vertex has is $n$. If no vertex has $n$ chips, then if we consider the divisor $D-2(v_i)$ and burn with respect to $v_i$, the entire graph burns. If a vertex, $v_j$ has $n$ chips, then consider the divisor $D-(v_i)-(v_j)$; now no vertex has at least $n$ chips, so burning with respect to $v_i$ burns the entire graph. Finally, if we are in case (3), then the maximum number of chips a vertex has is $n+1$. If no vertex has at least $n$ chips, then we can consider the divisor $D-2(v_i)$, and burning with respect to $v_i$ burns the entire graph. If a vertex, $v_k$, has at least $n$ chips, then there are no other vertices with at least $n$ chips. We can then consider the divisor $D-(v_i)-(v_j)$. Notice that we can fire subsets of vertices to move $n$ chips to $v_i$; then, $D(v_i)=n-1$ and no vertex has at least $n$ chips. If we then burn with respect to $v_j$, the entire graph burns. Thus, a winning divisor cannot have fewer than $2n-1$ chips, and $\gon(B_{n,e})=2n-1$.

\end{proof}




The graphs $B_{n,e}$ can only yield pairs of first and second gonalities of the form $(a,2a)$ and $(a,2a-1)$ where $a\geq 2$, as well as $(1,2)$ from $B_{2,1}$; in fact, we already had the pair $(a,2a-1)$ from \(K_{a,a}\).  To obtain graphs with lower second gonalities, we construct another family.  For $a\leq b$, let $B^*_{a,b}$ be the generalized banana graph on $a$ vertices $\left\{v_1,\dots,v_a\right\}$ where the number of edges between $v_i$ and $v_{i+1}$ for $1\leq i\leq a-1$ is $b-a+i+1$. In other words, there are $b$ edges between the first pair of vertices starting from $\{v_{a-1},v_a\}$, with each subsequent pair having one fewer edge than the pair before as we move from $v_a$ to $v_1$. 


\begin{figure}[H]
\begin{center}
\includegraphics[scale=0.8]{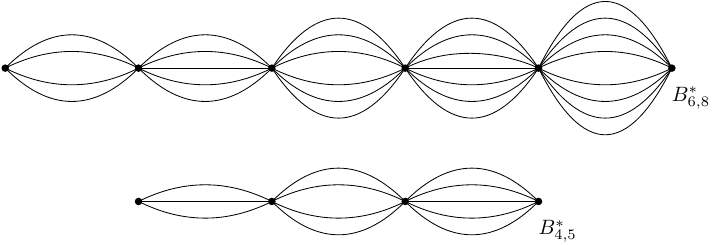}
\caption{The generalized banana graphs $B^*_{6,8}$ and $B^*_{4,5}$.}
\label{fig:banananas}
\end{center}
\end{figure}

\begin{lemma}\label{firstgonbanana}
We have $\gon_1(B^*_{a,b})=a$.
\end{lemma}

Our proof mimics that of Lemma 5 in \cite{ccNewtonPolygons}, in which it is proven that $\gon_1(B^*_{a,a})=a$. 

\begin{proof}
As with any graph, we have $\gon_1(B^*_{a,b})\leq |V(B^*_{a,b})|=a$. Suppose there exists a divisor $D$ such that $\deg(D)<a$ and $D$ has positive rank. We can also assume that $D$ is $v_1$-reduced, so $D(v_1)\geq 1$. There exists some other vertex with zero chips on it; let $i$ be the maximal index for which $D(v_i)=0$. We can then perform Dhar's burning algorithm with respect to $v_i$. The chips on the vertices $v_{i+1},\dots,v_m$ do not move, because $D$ is already $v_1$-reduced and fire from $v_1$ would pass through $v_i$. If $m$ is the number of edges between $v_{i-1}$ and $v_{i}$, then there need to be at least $m$ chips on the subgraph induced by the vertices $\left\{v_1,\dots,v_{i-1}\right\}$ since the entire graph will not burn as $D$ has positive rank. However, $m\geq i$, and $D(v_{i+1}),\dots,D(v_a)\geq 1$, so $\deg(D)=m+(a-i)\geq a$, a contradiction.
\end{proof}

\begin{lemma}\label{bananapairs}
Let $2\leq a\leq b\leq 2a-1$. We have $\gon_2(B^*_{a,b})=b+1$.
\end{lemma}

\begin{proof}
First notice that the divisor $(b+1)\cdot v_a$ has rank at least $2$: for any $k$ with $1\leq k\leq a$, this divisor is equivalent to $(b+k-a)\cdot v_k+\sum_{k\leq i\leq a}(v_k)$.  Allowing $k$ to vary, we find effective divisors greater than any given effective divisor of degree $2$. This means that $\gon_2(B^*_{a,b})\leq b+1$. Now suppose that a divisor $D\in\textrm{Div}_+^k(B_{a,b}^*)$ has rank at least $2$, where $k\leq b$. We proceed by cases.
\begin{itemize}
    \item[(1)]If $D(v_a)=1$, we can consider $D-2\cdot(v_a)$.  This divisor has $-1$ chips on $v_a$ and $k-1\leq b-1$ chips on the other vertices; no chips can move from $v_a$ to the rest of the graph without introducing debt, so this remains true as we try to eliminate debt via subset firing moves.  The only way to eliminate the debt on $v_a$ would be to fire a subset $W$ with $v_{a-1}\in W$, but this is only possible to do without introducing new debt if there are at least $b$ chips on $v_{a-1}$, a contradiction.
    
    \item[(2)]If $D(v_a)=0$, choose $v_i\in\textrm{supp}(D)$, and consider $D-(v_i)-(v_a)$.  The same argument from case (1) shows that the debt on $v_a$ cannot be eliminated, a contradiction.
    
    \item[(3)]If $D(v_a)=k$, then running Dhar's burning algorithm on the divisor $D-(v_{a-1})-(v_a)$ with respect to $v_{a-1}$ burns the entire graph, a contradiction.
    
    \item[(4)]If $D(v_a)=\ell$, $2\leq\ell\leq k-1$, we induct on $a$ to show that $D$ cannot win the second gonality game. As a base case, consider $B^*_{3,b}$, $3\leq b\leq 5$.  We might as well place only 2 chips on $v_3$, because placing any more does not allow us to chip fire from $v_3$ anyways. We can then consider the second gonality of the subgraph induced by $\left\{v_1,v_2\right\}$, which is the (usual) banana graph $B_{b-1}$. Since $\textrm{gon}_1(B_{b-1})=2$, we have $\gon_2(B_2)=3$ and $\gon_2(B_3)=4=\gon_2(B_4)$. Thus, no divisor $D$ with degree $\deg(D)\leq b$ can win the second gonality game if $2\leq D(v_3)\leq b-1$.
    
    Now suppose that for $B^*_{a,b}$, if a divisor $D$ has degree $\deg(D)=k\leq b$ and $2\leq D(v_a)\leq b-1$, then $D$ cannot win the second gonality game. Consider the graph $B^*_{a+1,b'}$ with $a+1\leq b'\leq 2a+1$, and suppose $D'$ is a  divisor with degree $\deg(D')=k\leq b'$ with rank at least $2$. Again, we can assume that $D'(v_{a+1})=2$. The subgraph $G'$ induced by the vertices $\left\{v_1,\dots,v_a\right\}$ is then $B^*_{a,b'-1}$, where $a\leq b'-1\leq 2a-1$.
    Restricting $D'$ to $G'$, we have  $\deg(D')=k-2\leq b-2\leq 2a-1$. From the cases above, as well as the inductive hypothesis, we know that $D'$ cannot win the second gonality game on $G'$. Thus, $D$ cannot win the second gonality game on $B^*_{a+1,b'}$.
\end{itemize}
Therefore, $\gon_2(B^*_{a,b})=b+1$.
\end{proof}

We can now prove that any reasonable pair of first two gonalities is achieved by some graph.

\begin{proof}[Proof of Theorem \ref{theorem:first_two}]  Let $m,n\in \mathbb{Z}$ with $m< n\leq 2n$.
To achieve the first and second gonality pair $(m,n)$, we can consider the graph $B^*_{m,n-1}$, which is well-defined since $m\leq n-1$. From Lemma \ref{firstgonbanana}, $\gon_1(B^*_{m,n-1})=m$. Notice that if $m+1\leq n\leq 2m$, then $m\leq n-1\leq 2m-1$, which are exactly our bounds for $a,b$ in Lemma~\ref{bananapairs}. Thus, given $(m,n)$ within our constraints, the graph $G=B^*_{m,n-1}$ has first gonality $\gon_1(G)=m$ and second gonality $\gon_2(G)=n$.
\end{proof}

We do not easily arrive at a corresponding result for the first three gonalities of a graph.  If $\ell=\gon_1(G)$, $m=\gon_2(G)$, and $n=\gon_3(G)$, then we certainly have $\ell<m<n$, $m\leq 2\ell$, and $n\leq \ell+m$.   However, not all triples $(\ell,m,n)$ satisfying these constraints are the first three gonalities of a graph; for instance,  due to  Proposition \ref{prop:hyperelliptic}, there exists no graph $G$ with $\gon_1(G)=2$, $\gon_2(G)=3$, and $\gon_2(G)=5$.  Determining which triples can be the first three gonalities of a graph seems to be an open question, although answering Question \ref{question:adding_graphs} might help.  For instance, an affirmative answer to that question would imply the triple $(3,5,8)$ to be impossible.  It could also be helpful to study the \emph{Clifford index} of graphs, defined as
\[\textrm{Cliff}(G):=\min\{\deg(D)-2r(D)\,|\, r(D)\geq 1\textrm{ and }\deg(D)\leq g-1\}.\]
As shown in \cite{cliffords_alg}, the Clifford index of an algebraic curve $C$ is always either $\gon_1(C)-2$ or $\gon_1(C)-3$; as noted in \cite{random_graphs}, it is unknown whether the Clifford index of graphs follows the same behavior.  If it does, then we could not have a graph with first three gonalities of the form $(a,a+1,a+2)$ if the graph has genus $g$ at least $a+3$, since then we would have $\textrm{Cliff}(G)\leq \gon_3(G)-2\cdot 3=a+2-6=a-4=\gon_1(G)-4$.

Another open question is the following:  given $m$ and $n$ as in Theorem \ref{theorem:first_two}, what are the possibilities for the genus $g$ of a graph $G$ with $\gon_1(G)=m$ and $\gon_2(G)=n$? In particular, when are there infinitely many possibilities for $g$? For instance, if $(m,n)=(2,3)$, then $g$ must be $1$; but if $(m,n)=(2,4)$, then $g$ can be any integer satisfying $g\geq 2$.  We could then ask for the minimum genus of a graph achieving a prescribed pair of first and second gonalities.  Since \(g(B^*_{a,b})=(a-1)\left(b-\frac{a}{2}\right)\), we can certainly achieve the pair \((a,b+1)\) with this genus.  However, this is not always the lowest possible genus.  For example, to achieve \((m,2m-1)\), we could use \(B^*_{m,2m-2}\) with genus \((m-1)(\frac{3}{2}m-2)\); or we could use either \(B_{m,m}\) or \(K_{m,m}\), both of which have a smaller genus of \((m-1)^2\).  (We remark that to achieve \((m,m+1)\) we could use either \(B^*_{m,m}\) or \(K_{m+1}\), both of which have genus \(\frac{m(m-1)}{2}\).)

An interesting direction for future work would be to also ask similar questions for algebraic curves; for instance, given \(m\) and \(n\) satisfying reasonable assumptions, what sorts of algebraic curves can we find with first gonality \(m\) and second gonality \(n\)?  A nice starting point for this exploration could be in the setting of plane curves.  For instance, let \(P\) be the triangle with vertices at \((0,0)\), \((n,0)\), and \((0,m)\) (where \(2\leq m< n\leq 2m\)),  let \(f(x,y)\) have support in \(P\), and let \(C\) be the curve defined by \(f\).  Then generically, it is known by \cite[Corollary 6.2]{linear_pencils} that \(\textrm{gon}_1(C)=m\); and it is conjectured in \cite[Conjecture 7.1]{lattice_size} that \(\textrm{gon}_2(C)=n\) if we only consider divisors whose corresponding divisorial map is birational onto its image.  (The first of these relies on the lattice width of \(P\) \cite[Theorem 4]{ccNewtonPolygons}, and the second on the lattice size of \(P\) with respect to the standard lattice triangle \cite[Theorem 3.5]{lattice_size}.)

\section{A Modified Burning Algorithm}
\label{section:modified_dhars}

In this section, we present a modified version of the iterated Dhar's burning algorithm suited for checking if a divisor has higher gonality at least $r$.  Although this can be accomplished using the traditional Dhar's algorithm with the same time complexity, our algorithm is more transparent when determining gonalities beyond the first, and so may be useful for proofs relying on the burning algorithm.  

Recall that our method for computing the first gonality of a graph relies on being able to check whether the rank of a divisor $D$ is at least $1$.  The burning algorithm also provides a method for checking if the rank of $D$ is at least $r$.
\begin{enumerate}
\item For each of the finitely many $E \in \Div_+^r(G)$, take $D - E$.
\item Choose any $v \in V(G)$, and semi-reduce $D-E$ with respect to $v$.  Then, use the burning algorithm to compute $\Red_v(D-E)$, and check if it is effective.  If so, continue with more choices of $E$; if not, $r(D) < r$.
\end{enumerate}
To check if $\gon_r(G) > k$, run the above two steps on every effective divisor $D$ of degree $k$.

Our modified algorithm (Algorithm \ref{algorithm:modified_dhars}) provides an alternative to step 2; this algorithm is recursive, calling itself with \(\textbf{Alg}\).  Rather than semi-reduction followed by a repeated burning process to reach a reduced divisor, it consists simply of a repeated burning process that determines whether or not the given divisor is equivalent to an effective divisor.  It is thus one of many algorithms that solves the \emph{dollar game}, which asks: given a divisor $D$, is there effective divisor $D'$ with $D\sim D'$?  For other algorithms solving the dollar game, see \cite[Chapter 3]{cp}.

\begin{algorithm}
\caption{Modified Dhar's Burning Algorithm}
\label{algorithm:modified_dhars}
\begin{algorithmic}
	\Require A divisor $D=D^+-D^-$, where $D^+,D^-\geq 0$.
	\Ensure A divisor $D' \in \abs{D}$ satisfying $D' \geq 0$, or \textsc{None} if none exists.
	\If{$D\geq0$}
	\State \Return $D$
    \EndIf
	\State $W \coloneqq V(G) \setminus\supp(D^-)$
	\While{$W \neq \emptyset$}
	    \If{$D(v) < \outdeg_W(v)$ for some $v \in V(G)$}
	    \State $W = W\setminus \{v\}$ \Comment{$v$ burns}
	    \Else
	        \State return $\textbf{Alg}(D - \Delta \mathbbm{1}_{W})$
        \EndIf
	\EndWhile
	\Return \textsc{None} \Comment{entire graph burned}
\end{algorithmic}
\end{algorithm}

We offer the following intuitive explanation of Algorithm \ref{algorithm:modified_dhars}.  Given $D=D^+-D^-$ where $D^+\geq 0$ and $D^->0$, set all vertices of $\supp(D^-)$ on fire.  Let the usual burning process propagate through the graph.  If the whole graph burns, then $D$ is not equivalent to any effective $D'$.  If the whole graph does not burn, fire the unburned vertices $W$, and run the process again.  We refer to each time we run through this burning process a \emph{pass} through our algorithm or an \emph{iteration} of our algorithm.

\begin{example}\label{example:modified_dhars}
As an example of running Algorithm \ref{algorithm:modified_dhars} on a divisor, consider the graph $G$ at the top of Figure \ref{figure:dhars_example}.  Let $D=-(a)+2\cdot(c)+7\cdot(f)-(g)-2\cdot(h)$; this divisor is illustrated on the bottom left.  We can write $D=D^+-D^-$, where $D^+=2\cdot(c)+7\cdot(f)$ and $D^-=(a)+(g)+2\cdot(h)$.  In accordance with Algorithm \ref{algorithm:modified_dhars}, we set the vertex set $\textrm{supp}(D^-)=\{a,g,h\}$ on fire.  The fire spreads until the whole graph burns except for $f$, so we fire $f$ to obtain the next divisor $D'=-(a)+(b)+2\cdot(c)+(d)+(e)+2\cdot(f)-(h)$.  We then set the vertices $a$ and $h$ on fire.  This time the vertices $b,c,d,e,$ and $f$ remain unburned, so these vertices are all fired, giving the divisor $D''=2\cdot(a)+2\cdot(g)+(h)$.  Since debt has been eliminated, the algorithm terminates.  If the whole graph had burned before debt was eliminated, we would have known that eliminating debt were impossible.

\begin{figure}[hbt]
    \centering
    \includegraphics[scale=1]{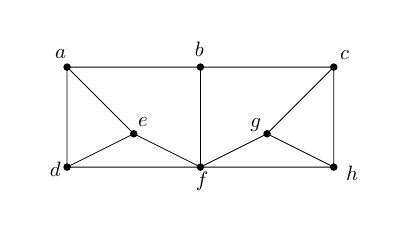}
    \\\includegraphics[scale=0.7]{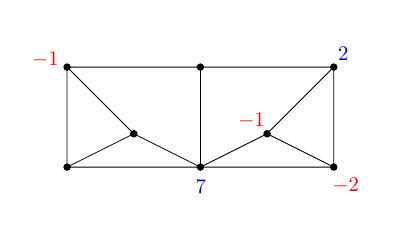}\qquad\includegraphics[scale=0.7]{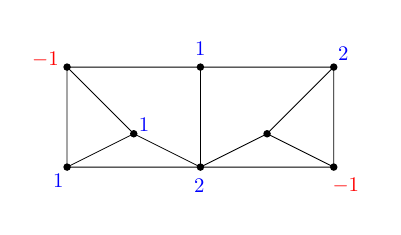}\qquad\includegraphics[scale=0.7]{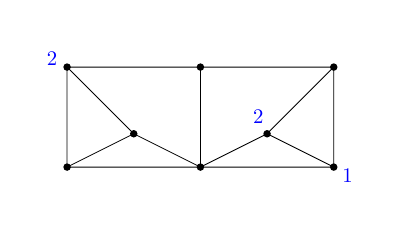}
    \caption{The graph $G$ and the steps of Algorithm \ref{algorithm:modified_dhars} from Example \ref{example:modified_dhars}}
    \label{figure:dhars_example}
\end{figure}

\end{example}

We need to argue that this algorithm terminates and that it is correct.  For the termination argument, we will use the $\beta_q(D)$ notation from Subsection \ref{subsection:usual_dhars}.  Recall that
\[\beta_q(D) = \paren{\sum_{v \in S_0} D(v), \sum_{v \in S_1} D(v),\dots, \sum_{v \in S_d} D(v)},\]
where $d=\textrm{diam}(G)$ and $S_i$ is the set of all vertices at distance $i$ from vertex $q$.  Thus $\beta_q(D)$ records in its $i^{th}$ component the number of chips at distance $i$ from $q$.

\begin{proposition}\label{prop:terminates}
Algorithm \ref{algorithm:modified_dhars} terminates.
\end{proposition}

\begin{proof}
Suppose for the sake of contradiction that the algorithm does not terminate on an input $D$.   Since the output of \textsc{None} is never returned, the algorithm is repeatedly firing subsets $W_0,W_1,W_2,\ldots$, giving an infinite sequence of equivalent divisors $D=D_0,D_1,D_2,\ldots$ where $D_{i+1} = D_i - \Delta \mathbbm{1}_{W_i}$.

Note that if $v\in W_i$, then $D_i(v)\geq 0$: otherwise,  $v$ would have been burned from the beginning.  Also note that if $v\in W_i$, then $\outdeg_{W_i}(v)\leq D_i(v)$, meaning that $D_{i+1}(v)\geq 0$. Thus, $D_i(v) \geq 0$ implies $D_j(v) \geq 0$ for all $j \geq i$.  Letting $T_i=\supp(D_i^-)$, this means we have
\[T_0\supseteq T_1\supseteq T_2\supseteq\cdots.\]
Since the $T_i$'s are all subsets of $V(G)$, this sequence must eventually stabilize at some index $k\geq 0$:
\[T_0\supset \cdots\supset T_k=T_{k+1}=T_{k+2}=\cdots\]
Let $T=T_k$.  We know $T\neq \emptyset$:  otherwise, the algorithm would return $D_k$.

We can assume without loss of generality that $k=0$; if not, then we can replace $D$ with $D_k$ and still have an input on which the algorithm fails to terminate.  Thus we have $\supp(D_i^-)=T$ for all $i$.

Consider the $|T|$-tuple of $(d+1)$-tuples $\left(\beta_u(D)\right)_{u \in T}$, which we can think of as a $|T|(d+1)$-tuple.  We can provide a partial order on such tuples by setting $(\beta_u(D'))_{u\in T} \geq (\beta_u(D))_{u\in T}$ if for all $u \in T$, we have $\beta_u(D') \geq \beta_u(D)$, where the latter is again taken with respect to the lexicographic ordering.  We claim that
\[(\beta_u(D_{i+1}))_{{u\in T}} > (\beta_u(D_i))_{{u\in T}} \quad \text{for all $i \geq 0$.}\]  
To argue this, it suffices to show that $\beta_u(D_{i+1}) > \beta_u(D_i)$ lexicographically for all $u\in T$. We can now replicate the argument from the usual Dhar's algorithm.  In particular, given $u\in T$, pick $w \in W_i$ such that $d(u, w)$ is minimized. Then $w$ has a neighbor $v$ such that $d(u, v) < d(u, w)$ and $v \notin W_i$ (where we might have $v = u$). This directly corresponds to $\beta_u(D_{i+1})$ being strictly greater than $\beta_u(D_i)$ with respect to the lexicographic ordering.

Thus we have an infinite strictly increasing sequence
\[(\beta_u(D_{0}))_{{u\in T}}<(\beta_u(D_{1}))_{{u\in T}}<(\beta_u(D_{2}))_{{u\in T}}<\cdots\]
of $|T|(d+1)$-tuples of integers.  Recall that each integer $\ell$ appearing in a tuple is the sum of the number of chips at a fixed distance from a vertex.  Since no new debt is introduced in the graph, we have $-\deg(D^-)\leq\ell\leq \deg(D^+)$.  Only finitely many $|T| (d+1)$-tuples of integers satisfy these bounds, a contradiction to the existence of an infinite sequence.  We conclude that the algorithm must terminate.
\end{proof}

Before we prove our algorithm is correct, we recall the following notation from \cite{db}. Let $D$ and $D'$ be equivalent divisors, with $D'=D-\Delta f$ for some function $f:V(G)\rightarrow\mathbb{Z}$.  Define $m=\max\{f(v)\,|\, v\in V(G)\}$ and $k=m-\min\{f(v)\,|\,v\in V(G)\}$.  The \emph{level set decomposition} of $f$ is the sequence of sets $A_0\subset A_1\subset \cdots\subset A_k=V(G)$ given by \[A_i=\{v\in V(G)\,|\, f(v)\geq m-i\}.\]
Thus, $A_i$ contains $v$ if and only if $f$ prescribes firing $v$ at least $m-i$ times.
The sequence of divisors $D_0,D_1,\ldots,D_k\in \abs{D}$ given by $D_0=D$ and $D_{i+1}=D_i-\Delta\one_{A_i}$ for all $i\in\{0,1,\ldots,k-1\}$ is called the \emph{divisor sequence associated with the level set decomposition}.  There is actually no harm in assuming \(k=m\):  if \(\min\{f(v)\,|\,v\in V(G)\}\geq 1\), redefine \(f\) to be \(f-\min\{f(v)\,|\,v\in V(G)\}\cdot \mathbbm{1}_{V(G)}\).  Since \(\Delta\mathbbm{1}_{V(G)}=0\), we still have  $D'=D-\Delta f$.  In this case, we then have that \(f=\sum_{i=1}^k \mathbbm{1}_{A_i}\), highlighting that we are indeed decomposing~\(f\).

\begin{proposition}\label{prop:correct}
Algorithm \ref{algorithm:modified_dhars} is correct.
\end{proposition}

\begin{proof}
If Algorithm  \ref{algorithm:modified_dhars} terminates because at some point $D\geq 0$, then we have found a sequence of chip-firing moves (namely the sequence of moves given by firing the subset $W$ produced by each pass through the algorithm) that transforms the initial divisor into an equivalent effective divisor.  Thus the algorithm is correct if it returns a divisor.

We now wish to argue that if  Algorithm \ref{algorithm:modified_dhars} terminates with \textsc{None}, then it is correct. Let $D$ and $D'$ be equivalent divisors, say with $D'=D-\Delta f$, where $D$ is not effective and $D'$ is.  Let $A_0\subset A_1\subset \cdots\subset A_k$ be the level set decomposition of \(f\) and $D=D_0,D_1,\ldots,D_k=D'$ the associated divisor sequence.

We recall \cite[Theorem 3.10]{db}:  for all $i\in\{0,1,\cdots,k\}$, we have $D_i(v)\geq\min\{D(v),D'(v)\}$ for all $v\in V(G)$. This means that no vertex goes any further into debt at any point during the sequence of firing moves. Suppose that $D=D^{+}-D^{-}$ where $D\sim D'\geq 0$ and choose $v\in \supp(D^{-})$.  Since $D(v)<D'(v)$, we have that $v\notin A_0$:  if $v$ were fired at every step, the number of chips on it would not increase.  Thus $A_0\cap \supp(D^{-})=\emptyset$.  It follows that $\textrm{outdeg}_{A_0}(v')\geq D(v')$ for all $v' \in A_0$.  Since $A_0$ is disjoint from $\supp(D^{-})$, it follows that when we run Algorithm \ref{algorithm:modified_dhars}, if $D\sim D'$ with $D'\geq 0$, then the whole graph does not burn, since at the very least $A_0$ will remain unburned. Contrapositively, if the whole graph burns, then $D$ is not equivalent to any $D'$ with $D'\geq 0$.  Thus  if the algorithm returns \textsc{None}, it is correct.
\end{proof}

We now bound the time complexity of Algorithm \ref{algorithm:modified_dhars}. We begin with the following lemma, whose statement and proof are very similar to that of Lemma 5 in \cite{tardos}.

\begin{lemma}\label{lemma:number_of_firings}
For any neighboring pair of vertices, the numbers indicating how many times each vertex has been fired so far cannot differ by more than $\deg(D^+)$ at any point during Algorithm \ref{algorithm:modified_dhars}.
\end{lemma}

\begin{proof}
Suppose that after some number of steps, we have reached $D'$ by performing $f(v)$ firing moves at each vertex $v$, so that $D'=D-\Delta f$.    Fix an edge $uv\in E(G)$.   If $f(u)=f(v)$, we are done.  Suppose $f(u)< f(v)$.  Let $U\subset V(G)$ be the set of vertices which were fired at most $f(u)$ times.  This means $u\in U$ and $v\notin U$.  By the definition of $U$, for any edge connecting $U$ with $V(G) - U$, more chips moved to $U$ than from $U$.  Hence, the total number of chips on $U$ has increased by at least $f(v)-f(u)$.  Since no new debt is created at any point in Algorithm \ref{algorithm:modified_dhars} and there are $\deg(D^+)$ chips at the beginning, we have $f(v)-f(u)\leq \deg(D^+)$.  A symmetric argument holds in the case that $f(v)<f(u)$.
\end{proof}

\begin{corollary}\label{corollary:modified_dhars_runtime}
The runtime of Algorithm \ref{algorithm:modified_dhars} is $O(|V(G)|^3\textrm{diam}(G)\deg(D^+))$.
\end{corollary}

\begin{proof}
Each pass through the algorithm runs in $O(|V(G)|^2)$ time: in the worst case, the vertices of $W$ burn one-by-one so we need to check at most $|V(G)|$ vertices at most $|V(G)|$ times each.

By construction, there is at least one vertex $q\in V(G)$ that does not fire over the course of the whole algorithm, since at least one vertex will be in debt the entire time (except possibly when it is brought out of debt at the very end).  By Lemma \ref{lemma:number_of_firings}, no neighbors of $q$ can fire more than $\deg(D^+)$ times; and therefore no neighbors of those vertices can fire more than $2\deg(D^+)$ times; and so on.  In general, no vertex of the graph can be fired more than $\textrm{diam}(G)\deg(D^+)$ times.  At least one vertex will fire on each pass through the algorithm (or else the algorithm will terminate), meaning that the total number of passes through the algorithm is at most $|V(G)|\textrm{diam}(G)\deg(D^+)$.  Combined with the running time of each individual pass, we have the desired result.
\end{proof}

It is worth noting that the time complexity computed in Corollary \ref{corollary:modified_dhars_runtime} is no better than that of Algorithm \ref{originaldhars}.  One upper bound for the runtime of that algorithm (ignoring the number of chips), presented as Algorithm 4 in \cite{bs}, is $3(|V(G)|-1)\textrm{diam}(G)\sum_{v\neq q}\deg(v)$. Since $\sum_{v\neq q}\deg(v)$ is $O(|V(G)|^2)$, this essentially matches our bound.  The work done in  \cite{bs} actually provides a more general bound on the runtime of the existing algorithm, derived by considering potential theory on graphs, of which the bound we recall here is only a corollary. It would be interesting to analyze the runtime of Algorithm \ref{algorithm:modified_dhars} through such a lens in future work.

We now describe a brute-force algorithm to determine if $\gon_r(G)>k$.

\begin{itemize}
    \item Choose a divisor $D\in \textrm{Div}_+^k(G)$ and a divisor $E\in  \textrm{Div}_+^r(G)$.  Run Algorithm \ref{algorithm:modified_dhars} on $D-E$. If it returns a divisor, move on to a new divisor $E$.  Keep testing $D-E$ for all possible divisors $E \in  \textrm{Div}_+^r(G)$ until either \textsc{None} is returned or until all possible choices of $E$ have produced a divisor.  If $D-E$ is equivalent to an effective divisor for every choice of $E$, return \textsc{False}, since the $r^{th}$ gonality is at most $\deg(D)=k$.

    \item If running Algorithm \ref{algorithm:modified_dhars} on $D-E$ returns \textsc{None}, then we know $r(D)<r$, so we can move on to a new divisor $D$ and repeat this process.  We will eventually either find $D$ with $r(D)\geq r$ (and return \textsc{False}) or we will find that there exists no such $D$ (and return \textsc{True}).

\end{itemize}

Setting $n=|V(G)|$, we have $|\textrm{Div}_+^k(G)|=\binom{n+k-1}{k}$ and  $|\textrm{Div}_+^r(G)|=\binom{n+r-1}{r}$. The runtime of Algorithm \ref{algorithm:modified_dhars} on $D-E$ is $O(n^3\textrm{diam}(G)\deg((D-E)^+))=O(n^3\textrm{diam}(G)k)$. Thus, the algorithm we have presented to determine if $\gon_r(G) > k$ will run in
\[O\left(\binom{n+k-1}{k}\binom{n+r-1}{r}n^3\textrm{diam}(G)k\right)\]
time.  For fixed $r$ and $k$, we can write this as
\[O\left(n^{k+r+3}\textrm{diam}(G)k\right),\]
or more concisely as
\[O\left(n^{k+r+4}k\right)\]
since $\textrm{diam}(G)\leq n-1$. This means that for fixed $k$ and $r$, there exists an algorithm, polynomial in $|V(G)|$, for determining if $\gon_r(G)>k$.

Of course, in order to compute $\gon_r(G)$, we would need to run this algorithm for numerous values of $k$, requiring an additional factor of $O(g(G))$ in the worst case. This quickly gives us a huge blow-up in computational time.  In the case of $r=1$, it was shown in \cite{gij} that it is NP-hard to bound gonality by $k$ (with $k$ no longer fixed).  We expect that the same holds for larger values of $r$, although we do not know of any work in this direction.

Experimentally, we find that using our modified algorithm to compute higher gonalities provides a modest improvement in real-time performance against the traditional approach, which is to $q$-semi-reduce and then compute the $q$-reduced divisor, over varying divisors $D-E$.  We also found that performance gains appear to increase with the number of vertices in the graph (see Figure \ref{figure:dhars_runtime_comparison}).  To conduct this analysis, we generated random connected graphs by fixing the number of vertices, inserting edges with probability $p = 0.5$, and excluding disconnected graphs. The two algorithms were compared on the same subset of 19 randomly generated connected graphs for each fixed number of vertices. Our implementation was in \texttt{Sage} \cite{sagemath}, and is available as supplementary material to this paper.  Also available is the data of our runtimes; we ran our code on a Lenovo ThinkPad X1 Carbon 4th edition, Intel Core i5-6200U, with 8GB of RAM.   We remark that most of the savings in our data come from the fact that the original iterated Dhar's algorithm spends time $q$-reducing even after debt is eliminated (without this issue, the two plots should be identical for \(r=1\)).  Running an early-return iterated Dhar's algorithm that terminates if debt is eliminated, we found a nearly identical performance to our modified algorithm.

\begin{figure}[hbt]
    \centering
\begin{tikzpicture}
\begin{axis}[
    title={\(r=1\)},
    xlabel={Number of Vertices},
    ylabel={Average Time Taken (in sec)},
    xmin=3.5, xmax=10.5,
    ymin=-0.1, ymax=2.6,
    xtick={4,5,6,7,8,9,10},
    ytick={0,0.5,1,1.5,2,2.5},
    legend pos=north west,
    ymajorgrids=true,
    grid style=dashed,
]

\addplot[
    color=blue,
    mark=triangle,
    ]
    coordinates {
    (4,0.00103)(5,0.0027)(6,0.0083)(7,0.0644)(8,0.139)(9,0.836)(10,2.451)
    };
    
\addplot[
    color=red,
    mark=square,
    ]
    coordinates {
    (4,0.0013)(5,0.0027)(6,0.0079)(7,0.05265)(8,0.114)(9,0.626)(10,1.784)
    };
    
\end{axis}
\end{tikzpicture}  \begin{tikzpicture}
\begin{axis}[
    title={\(r=1\)},
    xlabel={Number of Vertices},
    ylabel={},
    xmin=3.5, xmax=10.5,
    ymin=-5, ymax=85,
    xtick={4,5,6,7,8,9,10},
    ytick={0,20,40,60,80},
    legend pos=north west,
    ymajorgrids=true,
    grid style=dashed,
]

\addplot[
    color=blue,
    mark=triangle,
    ]
    coordinates {
(4,0.0053)(5,0.0186)(6,0.1096)(7,0.439)(8,1.903)(9,12.184)(10,74.361)
    };
    
\addplot[
    color=red,
    mark=square,
    ]
    coordinates {
(4,0.00593)(5,0.022)(6,0.1073)(7,0.3901)(8,1.6586)(9,9.8717)(10,56.298)
    };    
\end{axis}
\end{tikzpicture} \begin{tikzpicture}
\begin{axis}[
    title={\(r=1\)},
    xlabel={Number of Vertices},
    ylabel={},
    xmin=3.5, xmax=10.5,
    ymin=-50, ymax=1050,
    xtick={4,5,6,7,8,9,10},
    ytick={0,200,400,600,800,1000},
    legend pos=north west,
    ymajorgrids=true,
    grid style=dashed,
]

\addplot[
    color=blue,
    mark=triangle,
    ]
    coordinates {
(4,0.0108)(5,0.0911)(6,0.4952)(7, 4.529)(8,17.5101)(9,128.959)(10,815.069)
    };
    
\addplot[
    color=red,
    mark=square,
    ]
    coordinates {
(4,0.0149)(5,0.1007)(6,0.4823)(7,3.4585)(8,14.5719)(9,93.5253)(10,608.388)
    };
    
\end{axis}
\end{tikzpicture}
    \caption{Comparison of original Dhar's algorithm and our modified algorithm on graphs of low gonality (original in blue triangles, modified in red squares)}
    \label{figure:dhars_runtime_comparison}
\end{figure}
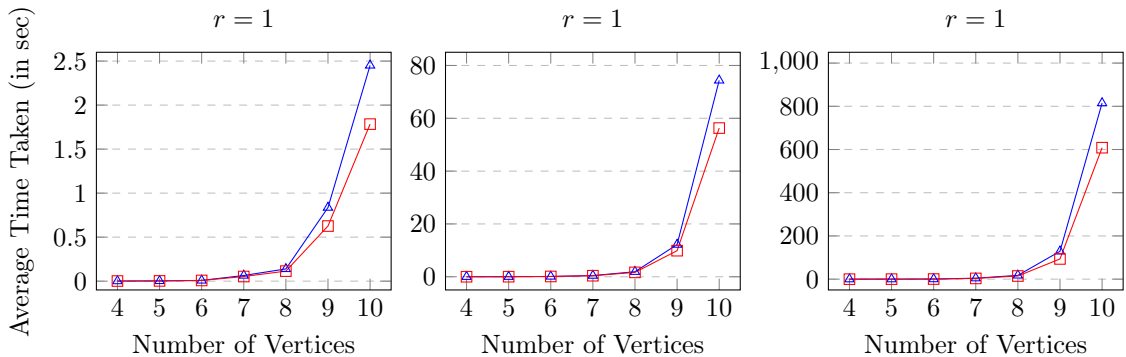

\medskip

\noindent \textbf{Acknowledgements}  The authors are grateful for the support they received from Williams College's SMALL REU, and from NSF Grants DMS1659037 and
DMS1347804.  They also thank two anonymous referees for many helpful comments.

\bibliographystyle{abbrv}
\bibliography{bibliography.bib}


\end{document}